	\documentclass[10pt, reqno]{amsart}

	\usepackage[left=1.25in, right=1.25in, top=1.5in, bottom=1.5in]{geometry}

	\usepackage{bm}

	\usepackage{MnSymbol}
	\usepackage{latexsym}
	\usepackage{amsthm}
	\usepackage{amscd}
	\usepackage{graphicx}
	\usepackage{enumitem}
	\usepackage[dvipsnames]{xcolor}
	\usepackage{tikz-cd}

	\usepackage{mathdots}

	\usepackage{microtype}

	\usepackage[mathscr]{euscript}

	\usepackage{import}

	\graphicspath{ {images/} }

\numberwithin{equation}{section}
\numberwithin{figure}{section}

\newtheorem{theorem}{Theorem}[section]
\newtheorem*{theorem*}{Theorem}

\newtheorem{proposition}[theorem]{Proposition}
\newtheorem*{proposition*}{Proposition}

\newtheorem{corollary}[theorem]{Corollary}

\newtheorem{lemma}[theorem]{Lemma}
\newtheorem*{lemma*}{Lemma}

\newtheorem{lemma-definition}[theorem]{Lemma-Definition}

\newtheorem{conjecture}[theorem]{Conjecture}
\newtheorem*{conjecture*}{Conjecture}

\theoremstyle{definition}
\newtheorem{definition}[theorem]{Definition}

\newtheorem{remark}[theorem]{Remark}
\newtheorem{construction}[theorem]{Construction}
\newtheorem{variant}[theorem]{Variant}

\newtheorem{example}[theorem]{Example}

\newtheorem{notation}[theorem]{Notation}
\newtheorem{warning}[theorem]{Warning}

\newtheorem{question}[theorem]{Question}

\newcommand{\step}[1] {\medskip \noindent {\em Step #1.\/}}

\newcommand{\Mod}{\mathrm{Mod}}

\newcommand{\K}{\mathrm{K}}

\makeatletter
\newcommand{\colim@}[2]{%
  \vtop{\m@th\ialign{##\cr
    \hfil$#1\operator@font colim$\hfil\cr
    \noalign{\nointerlineskip\kern1.5\ex@}#2\cr
    \noalign{\nointerlineskip\kern-\ex@}\cr}}%
}

\newcommand{\limit}{%
  \mathop{\mathpalette\varlim@{\leftarrowfill@\scriptscriptstyle}}\nmlimits@
}
\newcommand{\colimit}{%
  \mathop{\mathpalette\varlim@{\rightarrowfill@\scriptscriptstyle}}\nmlimits@
}
\makeatother

\newcommand{\A}{{\mathbf A}}
\newcommand{\C}{\mathbf C}

\newcommand{\G}{{\mathbf G}}

\renewcommand{\P}{{\mathbf P}}
\newcommand{\Q}{{\mathbf Q}}

\newcommand{\Z}{{\mathbf Z}}
\newcommand{\V}{{\mathbf V}}

\newcommand{\E}{{\mathrm{E}}}

\newcommand{\Ri}{\mathrm{R}}

\newcommand\scc{\mathscr C}

\newcommand\sce{\mathscr E}

\newcommand\sch{\mathscr H}

\newcommand\sco{\mathscr O}

\newcommand\scu{\mathscr U}

\newcommand\scx{\mathscr X}

\newcommand{\op}{\operatorname}

\newcommand{\Sp}{\op{Sp}}

\newcommand{\ev}{\operatorname{ev}}

\newcommand{\Dperf}{\mathrm{D}_{\mathrm{perf}}}

\DeclareMathOperator{\rk}{rk}
\DeclareMathOperator{\Char}{char}

\DeclareMathOperator{\Hom}{Hom}

\DeclareMathOperator{\Ker}{Ker}
\DeclareMathOperator{\Coker}{Coker}

\DeclareMathOperator{\Spec}{Spec}

\DeclareMathOperator{\Ext}{Ext}

\DeclareMathOperator{\Pic}{Pic}

\DeclareMathOperator{\NS}{NS}

\DeclareMathOperator{\Fun}{Fun}

\DeclareMathOperator{\ind}{ind}
\DeclareMathOperator{\hind}{ind_{H}}
\DeclareMathOperator{\etind}{ind_{\acute{e}t}}
\DeclareMathOperator{\per}{per}

\DeclareMathOperator{\Br}{Br}

\DeclareMathOperator{\Hdg}{Hdg}

\newcommand{\ch}{\mathrm{ch}}

\def\sHom{\mathop{\mathscr{H}\!\mathit{o}\! \kern .4pt \mathit{m}}\nolimits}

\def\sAut{\mathop{\mathscr{A}\!\mathit{ut}}\nolimits}
\def\sInn{\mathop{\mathscr{I}\! \kern .8pt \mathit{nn}}\nolimits}

\newcommand{\an}{\mathrm{an}}

\newcommand{\Wi}{\mathrm{W}}

\newcommand{\trdeg}{\mathrm{tr.deg}}

\newcommand{\Ho}{\mathrm{H}}

\newcommand{\HH}{\mathrm{HH}}
\newcommand{\HN}{\mathrm{HN}}

\newcommand{\et}{\mathrm{\acute{e}t}}

\newcommand{\bmu}{\bm{\mu}}

\renewcommand{\subset}{\subseteq}

\renewcommand{\cong}{\simeq}

\newcommand{\nc}{\newcommand}
\nc{\p}[2]{\frac{\partial #1}{\partial #2}}
\nc{\pp}[2]{\frac{\partial^2 {#1}} {\partial {#2} ^2}}
\nc{\pmix}[3]{\frac{\partial^2 {#1}}{\partial {#2}\, \partial {#3}}}

\newcommand{\bpm}{\begin{pmatrix}}
\newcommand{\epm}{\end{pmatrix}}

\newcommand{\bbm}{\begin{bmatrix}}
\newcommand{\ebm}{\end{bmatrix}}

	\usepackage{cite}
	\usepackage{hyperref}

	\hypersetup{colorlinks = true, 
				linktocpage = true, 
				linkcolor = MidnightBlue,citecolor=MidnightBlue,filecolor=MidnightBlue,urlcolor=MidnightBlue} 

	\newcommand{\dperf}{\mathrm{D}_{\mathrm{perf}}}
	\newcommand{\dcoh}{\mathrm{D}_{\mathrm{coh}}^b}
	\newcommand{\dqc}{\mathrm{D}_{\mathrm{qc}}}
	\newcommand{\Cat}{\mathrm{Cat}}
	
	\newcommand{\Map}{\mathrm{Map}}
	
	\newcommand{\B}{\mathrm{B}}

	\def\sAut{\mathop{\mathscr{A}\! \kern .8pt \mathit{ut}}\nolimits}
	\def\sIso{\mathop{\mathscr{I}\! \kern .9pt \mathit{so}}\nolimits}
	\def\sGrpd{\mathop{\mathscr{G}\! \kern 1.6pt \mathrm{rpd}}\nolimits}
	\def\sGrp{\mathop{\mathscr{G}\! \kern 1.6pt \mathrm{rp}}\nolimits}
	\def\sPr{\mathop{\mathscr{P}\! \kern 1.6pt \mathrm{r}}\nolimits}
	\def\sPic{\mathop{\mathscr{P}\! \kern 1.6pt \mathrm{ic}}\nolimits}
	\def\sM{\mathop{\mathit{s}\! \kern 1.6pt \mathscr{M}}\nolimits}
	\def\sCat{\mathop{\mathscr{C}\! \kern .9pt \mathit{at}}\nolimits}
	\def\sHH{\mathop{\mathscr{H}\! \kern .9pt \mathscr{H}}\nolimits}

	\newcommand{\ktop}{\mathrm{K}^{\mathrm{top}}}
	\newcommand{\wtop}{\mathrm{W}^{\mathrm{top}}}
	
	{}
	\newcommand{\HP}{\mathrm{HP}}
	
	\newcommand{\tors}{\mathrm{tors}}

	\newcommand{\wdots}{\kern 1.8pt \cdots \kern 1.8pt}
	\newcommand{\topo}{\mathrm{top}}
	
	\newcommand{\alg}{\mathrm{alg}}

	\newcommand{\KU}{\mathrm{KU}}
	\newcommand{\cs}{\mathrm{cs}}
	\newcommand{\rig}{\mathrm{rig}}
	\newcommand{\cont}{\mathrm{cont}}
	
	\newcommand{\rP}{\mathrm{P}}
	\newcommand{\CP}{\mathbf{CP}}
	\newcommand{\dBr}{\mathrm{dBr}}
	\newcommand{\sh}{\mathrm{sh}}

	\newcommand{\Shv}{\mathrm{Shv}}

	\newcommand{\rV}{\mathrm{V}}

	\newcommand{\sK}{\underline{\mathrm{K}}}

	\usepackage{etoolbox}
	\makeatletter
	\patchcmd{\@adjustvertspacing}
	  {\jot\baselineskip \divide\jot 4}
	  {\jot=6pt}
	  {}{}
	\makeatother

	\setlist{topsep=4pt plus 2pt minus 5pt, itemsep=8pt plus 2pt minus 5pt}

\begin{document}

	\title{Hodge theory of twisted derived categories and the period-index problem}
	\author{James Hotchkiss}
	\address{Department of Mathematics, University of Michigan, Ann Arbor, MI 48109 \smallskip}
	\email{htchkss@umich.edu}

	\begin{abstract}
	We study the Hodge theory of twisted derived categories and its relation to the period-index problem. Our main contribution is the development of a theory of twisted Mukai structures for topologically trivial Brauer classes on arbitrary smooth proper varieties and in families. As applications, we construct Hodge classes whose algebraicity would imply period-index bounds; construct new counterexamples to the integral Hodge conjecture on Severi--Brauer varieties; and prove the integral Hodge conjecture for derived categories of Deligne--Mumford surfaces. 
	\end{abstract}

	\maketitle

	\tableofcontents

	\setlength{\abovedisplayskip}{8pt plus 2pt minus 5pt}
	\setlength{\belowdisplayskip}{8pt plus 2pt minus 5pt}
	\setlength{\abovedisplayshortskip}{0pt plus 3pt}
	\setlength{\belowdisplayshortskip}{6pt plus 3pt minus 5pt}
	\setlength{\parskip}{1pt}

	\section{Introduction} 
	\label{sec:introduction}

	\subsection{The period-index problem} Let $K$ be a field. The Brauer group $\Br(K)$ is the group of equivalence classes of central simple algebras over $K$ with tensor product, where $A \sim A'$ if there exists an isomorphism of matrix algebras $\mathrm{M}_r(A) \simeq \mathrm{M}_{s}(A')$, for integers $s, r > 0$. Each Brauer class $\alpha$ is represented by a unique division algebra $D$, with every central simple algebra of class $\alpha$ isomorphic to a matrix algebra over $D$.

	There are two basic measures for the complexity of a Brauer class $\alpha$: the period $\per(\alpha)$, which is the order of $\alpha$ in $\Br(K)$, and the index $\ind(\alpha) = \sqrt{\dim D}$, where $D$ is the unique division algebra of class $\alpha$. From the structure theory of central simple algebras, $\per(\alpha)$ divides $\ind(\alpha)$, and $\per(\alpha)$ and $\ind(\alpha)$ have the same prime factors. The \emph{period-index problem} is the problem of determining an integer $\epsilon$ such that $\ind(\alpha)$ divides $\per(\alpha)^{\epsilon}$. We shall be concerned with the following longstanding conjecture:

	\begin{conjecture}[Period-index conjecture]
		Let $k$ be an algebraically closed field, and let $K/k$ be an extension of transcendence degree $d$. For any $\alpha \in \Br(K)$,
		\[
			\ind(\alpha) \divides \per(\alpha)^{d - 1}.
		\]
	\end{conjecture}

	\noindent
	The period-index conjecture is trivial for $d = 0$, and the case $d = 1$ follows from Tsen's theorem. For $d = 2$, the conjecture is known by work of de Jong \cite{dejong_period_index} when $\per(\alpha)$ is prime to the characteristic of $k$, and Lieblich \cite{lieblich_twisted_sheaves} and de Jong--Starr \cite{dejong_starr} in general. For $d \geq 3$, there is not a single example of a field $K$ for which a uniform period-index bound is known.

	\subsection{Hodge theory of twisted derived categories}

	The goal of the present paper is to introduce new Hodge-theoretic methods into the study of the period-index conjecture for function fields over $\C$, using a framework recently introduced by Perry \cite{Perry_CY2}.

	Let $X$ be a smooth, proper variety over $\C$, and let $\scc \subset \dperf(X)$ be an admissible subcategory. Perry associates to $\scc$ a finitely generated abelian group $\ktop_0(\scc)$, equipped with a Hodge structure of weight $0$. For example, $\ktop_0(\dperf(X))$ coincides with complex topological $K$-theory $\ktop_0(X)$, and the Hodge structure on $\ktop_0(X)$ is determined by the condition that the Chern character
	\[
		\ch:\ktop_0(X) \to \bigoplus_{k} \Ho^{2k}(X^{\an}, \Q(k))
	\]
	induces an isomorphism of $\Q$-Hodge structures, after extending scalars on the left.

	In general, for an admissible subcategory $\scc \subset \dperf(X)$, the Hodge structure $\ktop_0(\scc)$ is a summand of $\ktop_0(X)$. Moreover, $\ktop_0(\scc)$ admits a natural map from $\K_0(\scc)$, which factors through the subgroup $\Hdg(\scc, \Z)$ of integral Hodge classes. The \emph{integral Hodge conjecture} for $\scc$ asserts that the homomorphism
	\[
		\K_0(\scc) \to \Hdg(\scc, \Z) \subset \ktop_0(\scc)
	\]
	is surjective. When $X$ is smooth and proper, the integral Hodge conjecture for $\dperf(X)$ is closely related, but not equivalent, to the integral Hodge conjecture for $X$ in all degrees. 

	If $X$ is a smooth, proper variety (or, more generally, an orbifold) and $\alpha \in \Br(X)$, then the twisted derived category $\dperf(X, \alpha)$ fits into the formalism above, and we obtain a weight $0$ Hodge structure $\ktop_0(\dperf(X, \alpha))$. Crucially, the index of $\alpha$ may be computed as the positive generator of the image of the rank homomorphism from algebraic $K$-theory,
	\[
		\rk:\K_0(\dperf(X, \alpha)) \to \Z.
	\]
	As a consequence, the period-index conjecture for $\alpha$ may be divided into two parts:
	\begin{enumerate} [label = (\arabic*)]
		\item \label{pt:construct} Construct a Hodge class $v \in \ktop_0(\dperf(X, \alpha))$ of rank $\per(\alpha)^{\dim X - 1}$.
		\item \label{pt:ihc} Prove that $v$ is algebraic.
	\end{enumerate}
	The principal goal of the present paper to study the Hodge structure $\ktop_0(\dperf(X, \alpha))$, which, it turns out, may be computed whenever the Brauer class is topologically trivial. As a byproduct, we will solve step \ref{pt:construct} when $\per(\alpha)$ is prime to $(\dim X - 1)!$ and $\alpha$ is topologically trivial.

	Following the strategy indicated above, we prove the period-index conjecture for Brauer classes on complex abelian threefolds in forthcoming work with Perry.

	\subsection{Twisted Mukai structures}

	Our main technical result is the computation of $\ktop(\dperf(X, \alpha))$ in terms of twisted Mukai structures. Since the work of Huybrechts and Stellari \cite{huybrechts_stellari}, twisted Mukai structures  have played an important role in the theory of equivalences between twisted derived categories in general, and twisted K3 surfaces in particular. 

	We first recall the notion of a topologically trivial Brauer class. Supposing for simplicity that $X$ is projective, there is an exact sequence
	\[
		\begin{tikzcd}
			\Ho^2(X^{\an}, \Q(1)) \ar[r, "\exp"] & \Br(X) = \Ho^2(X^{\an}, \sco_{X^{\an}}^\times)^{\tors} \ar[r] & \Ho^3(X^{\an}, \Z(1))^{\tors},
		\end{tikzcd}
	\]
	and a Brauer class $\alpha \in \Br(X)$ is \emph{topologically trivial} if it lies in the kernel of the right-hand map. Equivalently, $\alpha$ is topologically trivial if it admits a \emph{rational $B$-field}: a class $B \in \Ho^2(X^{\an}, \Q(1))$ with $\exp(B) = \alpha$.

	From a rational $B$-field, one defines a $B$-twisted Mukai structure as the unique Hodge structure of weight $0$ with underlying abelian group $\ktop_0(X)$ such that the $B$-twisted Chern character
	\[
		\ktop_0(X)^B \to \bigoplus_{k} \Ho^{2k}(X^{\an}, \Q(k)), \hspace{5mm} v \mapsto \exp(B) \cdot \ch(v)
	\]
	induces an isomorphism of $\Q$-Hodge structures, after extending scalars on the left. For example, a class $v \in \ktop_0(X)$ is a Hodge class for the twisted Mukai structure $\ktop_0(X)^B$ if and only if
	\[
		\exp(B) \cdot \ch(v) = (\rk(v), \ch_1(v) + \rk(v) \cdot B, \dots )
	\]
	is Hodge in each degree.

	\begin{theorem}
	\label{thm:rational_b_field_intro}
	    If $\alpha$ is topologically trivial, then there is an isomorphism of Hodge structures
	    \[
	    	\ktop_0(\dperf(X, \alpha)) \simeq \ktop_0(X)^B,
	    \]
	    where $B$ is any rational $B$-field for $\alpha$.
	\end{theorem}

	There is an analogous result for the weight $1$ Hodge structure on odd topological $K$-theory. In addition, we develop a theory of variations of twisted Mukai structures, and prove an analogue of Theorem~\ref{thm:rational_b_field_intro} for families of twisted derived categories. Finally, we also obtain an extension to the case of smooth, quasi-projective varieties, which can be used to some extent to treat the case when $\alpha$ is not topologically trivial.

	\subsection{The Hodge-theoretic index}

	We define the \emph{Hodge-theoretic index} $\hind(\alpha)$ of $\alpha$ to be the positive generator of the image of the rank homomorphism
	\[
		\Hdg(\dperf(X, \alpha), \Z) \to \Z.
	\]
	The Hodge-theoretic index is closely related to both the integral Hodge conjecture for $\dperf(X, \alpha)$ and the period-index conjecture for $\alpha$. Moreover, it is connected with previous work of Kresch \cite{kresch_quarternion} on quaternion algebras, and to the \'etale index introduced by Antieau \cite{antieau_et_ind}.

	The Hodge-theoretic index satisfies a number of useful properties, including the following divisibility:
	\[
		\per(\alpha) \divides \hind(\alpha) \divides \ind(\alpha).
	\]
	In general, it may differ from both the period and the index. The period-index conjecture implies that $\hind(\alpha)$ divides $\per(\alpha)^{\dim X - 1}$, and using the isomorphism from Theorem~\ref{thm:rational_b_field_intro} (which preserves rank), we prove the following:

	\begin{theorem}
	\label{thm:hind_bound_intro}
	    Let $\alpha \in \Br(X)$ be topologically trivial. Then 
	    \[
	    	\hind(\alpha) \divides \per(\alpha)^{\dim X - 1} \cdot ((\dim X - 1)!)^{\dim X - 2}.
	    \]
	\end{theorem}

	In particular, if $\per(\alpha)$ is prime to $(\dim X - 1)!$, then $\hind(\alpha)$ divides $\per(\alpha)^{\dim X - 1}$. If, in addition, the integral Hodge conjecture holds for $\dperf(X, \alpha)$, then the period-index conjecture holds for $\alpha$.

	\subsection{Potential obstructions to period-index bounds}
	When $\per(\alpha)$ divides $(d - 1)!$, it is not clear that one can construct Hodge classes of rank $\per(\alpha)^{\dim X - 1}$. The first interesting case occurs for a topologically trivial Brauer class $\alpha$ of period $2$ on a threefold $X$, and one may extract from Theorem~\ref{thm:rational_b_field_intro} a necessary condition for the truth of the period-index conjecture:

	\begin{theorem}
	\label{thm:obstruction_intro}
	    Let $X$ be a smooth, proper threefold over $\C$, with $v \in \Ho^2(X^{\an}, \Z(1))$. If the index of $\alpha = \exp(v/2)$ divides $4$, then there exists $H \in \NS(X)$ such that
	    \[
	    	v^2 + v \cdot H \in \Ho^4(X^{\an}, \Z(2))
	    \]
	    is congruent to a Hodge class modulo $2$.
	\end{theorem}

	We do not know if the conclusion of Theorem~\ref{thm:obstruction_intro} is satisfied for arbitrary degree $2$ cohomology classes on threefolds. The stronger condition that $v^2$ is congruent to a Hodge class modulo $2$ is Kresch's obstruction \cite{kresch_quarternion} to $\ind(\alpha) = 2$. Theorem~\ref{thm:hind_bound_intro} and Theorem~\ref{thm:obstruction_intro} may be regarded as Hodge-theoretic counterparts to results of Antieau and Williams \cite{antieau_williams_6_complex}, \cite{top_per_ind} on the topological period-index conjecture.

	\subsection{Counterexamples to the integral Hodge conjecture}

	The Hodge-theoretic index provides a new method for constructing counterexamples to the integral Hodge conjecture for varieties and categories. The idea is to begin with a Brauer class $\alpha$ for which one can compute the index (for instance using the method of Gabber \cite[Appendice]{CT_gabber}), and then to show that $\hind(\alpha) < \ind(\alpha)$ by an explicit calculation in a twisted Mukai structure.

	\begin{theorem}
	\label{thm:gabber_intro}
	    Let $C$ be a curve of genus $g \geq 2$, and let $E_1, \dots, E_k$ be elliptic curves for $2 \leq k \leq g$. Suppose that $C, E_1, \dots, E_k$ are very general, and let $X = C \times \prod_{i = 1}^k E_i$. For each prime $\ell$, there is a Brauer class $\alpha_{\ell} \in \Br(X)[\ell]$ such that the integral Hodge conjecture fails for $\dperf(X, \alpha_{\ell})$. Moreover, for each Severi--Brauer variety $\P \to X$ of class $\alpha_{\ell}$, the integral Hodge conjecture fails for $\P$.
	\end{theorem}

	There is a classical counterpart to the method above, in which one constructs Hodge classes on a Severi--Brauer variety of class $\alpha$ whose algebraicity would imply that $\per(\alpha) = \ind(\alpha)$. We implement the method on Severi--Brauer varieties over abelian threefolds:

	\begin{theorem}
	\label{thm:sb_abelian_threefold}
	    Let $X$ be an abelian threefold, and let $\alpha \in \Br(X)$ be a Brauer class of period $2$ with $\ind(\alpha) > 2$. Then the integral Hodge conjecture fails for any Severi--Brauer variety $\P \to X$ of class $\alpha$. 
	\end{theorem}

	One may construct examples of abelian threefolds with a Brauer class of period $2$ and index $4$ using Gabber's method \cite[Appendice]{CT_gabber}, for instance. Note, however, that Theorem~\ref{thm:sb_abelian_threefold} does not imply that the integral Hodge conjecture fails for $\dperf(X, \alpha)$. In fact, we do not know an example of a 3-dimensional Calabi--Yau category for which the integral Hodge conjecture fails.

	\subsection{The integral Hodge conjecture for DM surfaces}

	The original counterexamples to the integral Hodge conjecture by Atiyah and Hirzebruch \cite{atiyah_hirz} are based on approximating the zero-dimensional Deligne--Mumford stack $\B((\Z/p)^{\times 3})$, for $p$ a prime. In contrast, it is not difficult to show that the integral Hodge conjecture holds for $\dperf(X)$, when $X$ is a smooth, proper Deligne--Mumford stack of dimension $\leq 1$ (Lemma~\ref{lem:ihc_curve_surjective}). As a final application of Theorem~\ref{thm:rational_b_field_intro}, we prove the following theorem:

	\begin{theorem}
	\label{thm:ihc_dm_intro}
	    Let $X$ be a smooth, proper Deligne--Mumford surface over $\C$. Then the integral Hodge conjecture holds for $\dperf(X)$.
	\end{theorem}

	In the case when the surface $X$ is a variety, the integral Hodge conjecture for $\dperf(X)$ follows from Lefchetz's (1,1) theorem and the degeneration of the Atiyah--Hirzebruch spectral sequence. Lefschetz's (1,1) theorem holds on smooth, proper Deligne--Mumford surfaces, but the topological $K$-theory of $\dperf(X)$ is quite different from the integral cohomology of $X$, as one may already see for $X = \B G$ with $G$ a finite group. 

	Instead, the proof of Theorem~\ref{thm:ihc_dm_intro} is based on a reduction to the case of a twisted derived category $\dperf(X, \alpha)$, where $X$ is an orbifold. Then, using Theorem~\ref{thm:rational_b_field_intro} (or more precisely, its extension to smooth, quasi-projective varieties), we ultimately reduce to the problem of producing an algebraic class of rank $\per(\alpha)$ in $\Hdg(\dperf(X, \alpha), \Z)$. The existence of such a class follows from---and in fact is equivalent to---de Jong's theorem that $\per(\alpha) = \ind(\alpha)$ \cite{dejong_period_index}.

	\subsection{Organization of the paper}

	In \S \ref{sec:twisted_derived_categories}, we review aspects of linear categories and the theory of twisted derived categories. In \S \ref{sec:hodge_theory_of_categories}, we review the Hodge theory of categories. In \S \ref{sec:twisted_mukai_structures}, we develop the theory of twisted Mukai structures on smooth projective varieties, in families, and on quasi-projective varieties, and prove Theorem~\ref{thm:rational_b_field_intro}. In \S \ref{sec:hodge_theoretic_index}, we study the Hodge-theoretic index, and prove Theorem~\ref{thm:hind_bound_intro} and Theorem~\ref{thm:obstruction_intro}. In \S \ref{sec:counterexamples_to_the_integral_hodge_conjecture}, we prove Theorem~\ref{thm:gabber_intro} and Theorem~\ref{thm:sb_abelian_threefold}. In \S \ref{sec:the_integral_hodge_conjecture_for_DM_surfaces}, we prove Theorem~\ref{thm:ihc_dm_intro}.

	\subsection{Conventions}

	For $\C$-linear categories, we have followed the conventions of \cite{Perry_NCHPD}. An \emph{orbifold} over a field $k$ is a Deligne--Mumford stack of finite type over $k$ with trivial generic stabilizers. We have followed Giraud's convention \cite[Example V.4.8]{giraud} for the Brauer class of a Severi--Brauer variety.

	\subsection{Acknowledgements} 

	It is a pleasure to thank Alex Perry for his advice and encouragement throughout the completion of this project. In addition, I would like to thank Bhargav Bhatt, Jack Carlisle, Andy Jiang, Johan de Jong, and Claire Voisin for useful conversations and correspondences. This work was partially supported by NSF DMS-1840234.

	\section{Twisted derived categories} 
	\label{sec:twisted_derived_categories}

	In \S \ref{ssec:linear_categories}, we review some terminology and conventions regarding linear categories. In \S \ref{ssec:twisted_derived_categories}, we review the basics of twisted derived categories. In \S \ref{ssec:non_abelian_gerbes}, we prove a key result on derived categories of gerbes (Proposition \ref{prop:non_abelian_reduction}). Finally, in \S \ref{ssec:root_stacks}, we discuss root stacks and their relation to derived categories, weak factorization, and extensions of Brauer classes.

	\subsection{Linear categories}
	\label{ssec:linear_categories}

	The content of \S\ref{sec:twisted_derived_categories} is rather formal, so we work in a greater level of generality than is necessary for the rest of the paper. Note, however, that for the main results of the paper, it is sufficient to consider smooth, separated Deligne--Mumford stacks over a field of characteristic $0$.

	In general, we follow the conventions and terminology of \cite{Perry_NCHPD}, with the exception that our base $S$ is permitted to be a \emph{perfect algebraic stack}, as in \cite{BZFN}. In contrast to \cite{BZFN}, we avoid for simplicity the language of derived algebraic geometry.

	\begin{definition}
	\label{def:perfect_stack}
	    An algebraic stack $S$ is \emph{perfect} if the following conditions hold:
		\begin{enumerate} [label = \normalfont{(\arabic*)}]
			\item $\dqc(S)$ is compactly generated. 
			\item The compact and perfect objects of $\dqc(S)$ coincide. 
			\item The diagonal of $S$ is affine.
		\end{enumerate}
	Quasi-compact tame Deligne--Mumford stacks with affine diagonal are perfect \cite{rydh_hall}. For example, separated Deligne--Mumford stacks of finite type over a field of characteristic $0$ are perfect.
	\end{definition}

	Let $S$ be a perfect algebraic stack. Then $\dperf(S)$, with the tensor product of complexes, may be regarded as a commutative algebra object of the category $\Cat_{\mathrm{st}}$ of small, idempotent-complete stable $\infty$-categories. 

	An \emph{$S$-linear category} is a $\dperf(S)$-module object of $\Cat_{\mathrm{st}}$. The collection of $S$-linear categories forms an $\infty$-category $\Cat_S = \Mod_{\dperf(S)}(\Cat_{\mathrm{st}})$, with a symmetric monoidal structure given by the tensor product of $\dperf(S)$-modules $\scc \otimes_{S} \scc'$. Moreover, given a $\scc, \scc' \in \Cat_S$, there is a mapping object $\Fun_{S}(\scc, \scc') \in \Cat_S$, satisfying the property that
	\[
	   	\Map_{\Cat_S}(\scc, \scc') = \Fun_S(\scc, \scc')^{\cong}
	\]   
	where the left-hand side denotes the morphism space in the category $\Cat_S$, and the right-hand side denotes the maximal $\infty$-subgroupoid of $\Fun_S(\scc, \scc')$, obtained from $\Fun_S(\scc, \scc')$ by discarding non-invertible $1$-morphisms.

	\begin{example}
		Let $f:X \to S$ be a morphism. Then $\dperf(X)$ is an $S$-linear category, with the action of $\dperf(S)$ given by 
		\[
			E \mapsto E \otimes f^*(F), \hspace{5mm} E \in \dperf(X), F \in \dperf(S).
		\]
	\end{example}

	Let $T \to S$ be a morphism, and let $\scc$ be an $S$-linear category. The \emph{base change $\scc_T$} is the tensor product
	\[
		\scc \otimes_{\dperf(S)} \dperf(T)
	\]
	regarded as an $T$-linear category. 

	\begin{example}
	 	Let $T \to S$ be a morphism between perfect algebraic stacks, and let $\scc = \dperf(S)$. Then
	 	\[
	 		\scc_T \cong \dperf(T)
	 	\]
	 	by \cite[Theorem 1.2]{BZFN}
	\end{example}

	\subsection{Twisted derived categories}
	\label{ssec:twisted_derived_categories}

	In this section, we establish some basic notation and results about twisted derived categories. We begin by briefly introducing the abstract theory, due to To\"en and extended by Antieau--Gepner \cite{Antieau-Gepner} and Antieau \cite{Antieau-et_tw}. 

	Let $X$ be a perfect stack. An $X$-linear category is a \emph{twisted derived category} over $X$ if there exists an fppf cover $U \to X$ such that $\scc_U$ is equivalent to $\Dperf(U)$. According to a result of To\"en \cite{Toen_der_azu}, any twisted derived category $\scc$ is determined up to equivalence by a class
	\[
		[\scc] \in \dBr(X) = \Ho^2_{\et}(X, \G_m) \times \Ho^1_{\et}(X, \Z),
	\]
	where $\dBr(X)$ is the \emph{derived Brauer group} of $X$. Briefly, one may construct $[\scc]$ as follows: the higher stack $\sIso_X(\scc, \dperf(X))$ of $X$-linear equivalences is a torsor under $G = \sAut_X(\dperf(X)) = \B \G_m \times \Z$, and is determined up to isomorphism by the cohomology class $[\scc] \in \Ho^1_{\et}(X, G)$. Conversely, for any $\alpha \in \dBr(X)$, there exists an $X$-linear category $\scc$, unique up to equivalence, with $\alpha = [\scc]$.
	
	\begin{notation}
		For $\alpha \in \dBr(X)$, we write $\dperf(X, \alpha)$ for the choice of a category $\scc$ with $[\scc] = \alpha$.
	\end{notation}

	We warn the reader that $\dperf(X, \alpha)$ is unique up to potentially non-unique $X$-linear equivalence. More precisely, if $\scc$ and $\scc'$ are twisted derived categories categories with $[\scc] = [\scc']$, then the set of isomorphism classes of $X$-linear equivalences $\scc \simeq \scc'$ is a torsor under $\Pic(X) \times \Ho^0_{\et}(X, \Z)$.

	\begin{example}
	\label{ex:dperf1}
		Let $\scx \to X$ be a $\G_m$-gerbe, and let $\dperf^1(\scx/X)$ be the derived category of perfect $1$-twisted complexes on $X$. When the gerbe structure on $\scx$ is understood, we simply write $\dperf^1(\scx)$. Then $\dperf^1(\scx)$ is a twisted derived category over $X$, with
		\[
			[\dperf^1(\scx)]= [\scx] \in \Ho^2_{\et}(X, \G_m) \subset \dBr(X).
		\]
		To prove this, observe that $\sIso_X(\dperf^1(\scx), \dperf(X))$ may be identified with $\sPic^{-1}(\scx/X) \times \Ho^0_{\et}(X, \Z)$, where the first factor is the stack of $-1$-twisted line bundles on $\scx/X$. Then $\sPic^{-1}(\scx/X) \to X$ is a $\G_m$-gerbe of class $[\scx]$ (cf. \cite[Prop. 2.11]{shin}). 
	\end{example}

	\begin{example}
	\label{ex:mun_gerbe}
		For $n > 0$, let $\scx \to X$ be a $\bmu_n$-gerbe. As in the previous example, the $X$-linear category $\dperf^1(\scx)$ of perfect $1$-twisted complexes on $\scx$ is a twisted derived category of class $\alpha \in \Ho^2_{\et}(X, \G_m)$, where $\alpha$ is the image of $[\scx]$ under the homomorphism
		\[
			\Ho^2_{\et}(X, \bmu_n) \to \Ho^2_{\et}(X, \G_m).
		\]
		In fact, there is an orthogonal decomposition
		\begin{equation}
		\label{eq:mun_gerbe_dec6}
			\dperf(\scx) = \langle \dperf^0(\scx), \dperf^1(\scx), \dots, \dperf^{n - 1}(\scx)  \rangle,
		\end{equation}
		where $\dperf^k(\scx)$ is the category of perfect $k$-twisted complexes on $\scx$, and $\dperf^0(\scx)$ is the pullback of $\dperf(X)$.
	\end{example}

	\begin{example}
	\label{ex:bernardara}
		Let $\pi:\P \to X$ be a Severi--Brauer variety of relative dimension $n - 1$. According to a result of Bernardara \cite{bernardara} (and \cite{bergh_gerbe} in the case of an algebraic stack $X$), there is an $X$-linear semiorthogonal decomposition 
		\begin{equation} \label{eq:bernardara_dec6}
			\dperf(\P) = \langle \dperf(\P)_0, \dperf(\P)_{1}, \dots, \dperf(\P)_{n - 1} \rangle,
		\end{equation}
		where $\dperf(\P)_0$ is the pullback of $\dperf(X)$, and for any $k$, $\dperf(\P)_k$ is a twisted derived category over $X$ of class $[\P]^{k}$, where $[\P]$ is the class of $\P$ in the cohomological Brauer group of $X$. 
	\end{example}

	\begin{warning}
		We warn the reader that we follow Giraud's convention \cite[Example V.4.8]{giraud} regarding the Brauer class of a Severi--Brauer variety, which differs from Bernardara's convention by a sign.
	\end{warning}

	\begin{remark}
		When $[\P] = 0$, \eqref{eq:bernardara_dec6} recovers Beilinson's semiorthogonal decomposition for a projective bundle, with $\dperf(\P)_k = \pi^*\dperf(X) \otimes \sco_{\P}(1)$ for any tautological line bundle $\sco_{\P}(1)$.
	\end{remark}

	\begin{remark}
		We explain the compatibility between Example~\ref{ex:bernardara} and Example~\ref{ex:mun_gerbe}. Let $X$ be a separated scheme of finite type over $\C$. Let $\scx \to X$ be a $\bmu_n$-gerbe, and suppose that there exists a Severi--Brauer variety $\P \to X$ which represents the image of $[\scx]$ in $\Ho^2_{\et}(X, \G_m)$.

		We observe that $\P_{\scx} \to \P$ is a projective bundle. Consider the pullback diagram
		\[
			\begin{tikzcd}
				\P_{\scx} \ar[d, "\pi'"] \ar[r, "f'"] & \P \ar[d, "\pi"] \\
				\scx \ar[r, "f"] & X
			\end{tikzcd}
		\]
		Let $\Pic^{-1}(\P_{\scx})_1$ be the set of relative hyperplane bundles on $\P_{\scx}/\scx$ which are $-1$-twisted for the gerbe structure $\P_{\scx} \to \P$. Then $\Pic^{-1}(\scx)_1$ is nonempty (cf. the proof of \cite[Theorem 6.2]{bergh_weak_factor}, for instance), and moreover it is a torsor under $\Pic(X)$. 

	\begin{lemma}
	\label{lem:weighted_lb_lemma}
	    For any $\sco_{\P_{\scx}}(1) \in \Pic^{-1}(\P_{\scx})_1$, the Fourier--Mukai functor
	    \[
	    	\Phi_{\sco_{\P_{\scx}(1)}}:\dperf^1(\scx) \to \dperf(\P)_1
	    \]
	    is an $X$-linear equivalence of categories. 
	\end{lemma}

	\begin{proof}
	   	From the semiorthogonal decompositions above, tensoring with $\sco_{\P_{\scx}}(1)$ induces an equivalence 
	    \[
	    	\dperf^1(\scx) \otimes_X \dperf(\P) \to \dperf^0(\scx) \otimes_X \dperf(\P),
	    \]
	    and an equivalence 
	    \[
	    	\dperf(\scx) \otimes_X \dperf(\P)_0 \to \dperf(\scx) \otimes_X \dperf(\P)_1,
	    \]
	    where each category is regarded as an admissible subcategory of $\P_{\scx}$ by pullback on both factors. Comparing the two, we see that tensoring with $\sco_{\P_{\scx}}(1)$ induces an equivalence
	    \[
	    	\dperf^1(\scx) \otimes_X \dperf(\P)_0 \to \dperf^0(\scx) \otimes_X \dperf(\P)_1,
	    \]
	    which may be identified with the Fourier--Mukai transform in the statement of the lemma.
	\end{proof}

	\end{remark}

	\subsection{Derived categories of non-abelian gerbes}
	\label{ssec:non_abelian_gerbes}

	Let $X$ be a perfect algebraic stack, and let $\scx \to X$ be a gerbe with finite inertia. Let $\hat \scx \to X$ be the moduli stack of simple coherent sheaves on $\scx \to X$, and let $\hat \scx^{\sh} \to X$ be the sheafification of $\hat \scx$ over $X$. We note that $\hat \scx \to \hat \scx^{\sh}$ is a $\G_m$-gerbe. 

	Let $\sce \in \dperf(\scx \times_X \hat \scx)$ be the universal simple, perfect sheaf, and let
	\[
		\Phi_{\sce}:\dperf(\scx) \longrightarrow \dperf^{1}(\hat \scx/\hat \scx^{\sh})
	\]
	be the Fourier-Mukai transform associated to $\sce$.

	\begin{proposition}
	\label{prop:non_abelian_reduction}
	    In the context above, if $X$ is a $\Q$-stack, then $\Phi_{\sce}$ is an equivalence. 
	\end{proposition}

	\begin{proof}
	    The question is fppf-local on $X$ \cite[Prop. 2.1, Prop. 2.2]{benzvi2012morita}. Since the statement is compatible with fppf base change, we ultimately reduce to the case when $X = \Spec k$, for $k$ an algebraically closed field with $\Char(k) = 0$, and $\scx = \B G$ for a finite constant group scheme $G$. Then $\hat \scx^{\sh}$ may be identified with the set $\{[V_1],\dots, [V_{\ell}]\}$ of isomorphism classes of irreducible $G$-representations. 

			In terms of $G$-representations, the component of $\Phi_{\sce}$ which maps to $\dperf(\{[V_i]\})$ is given by
		\[
			V \mapsto (V \otimes V_i)^G \cong \Hom_G(V^\vee, V_i), 
		\]  
		and it follows from Maschke's theorem and Schur's lemma that $\Phi_{\sce}$ gives an equivalence at the level of categories of coherent sheaves. 
	\end{proof}

	\begin{corollary}	
	\label{cor:non-abelian_gerbe_cover}
	    Let $X$ be a perfect algebraic stack over $\Q$, and let $\scx \to X$ be a gerbe with finite inertia. There exists a finite \'etale cover $Y \to X$ which is representable by algebraic spaces, a class $\alpha \in \Ho_{\et}^2(Y, \G_m)$, and an $X$-linear equivalence $\dperf(\scx) \cong \dperf(Y, \alpha)$.
	\end{corollary}

	\begin{proof}
	    We may take $Y = \hat \scx^{\sh}$ and $\alpha = [\hat \scx]$ in the context of Proposition~\ref{prop:non_abelian_reduction}.
	\end{proof}

	\begin{remark}
		When $\scx \to X$ is a $\bmu_n$-gerbe, $\hat \scx^{\sh}$ is a disjoint union of $n$ copies of $X$, labeled by the characters of $\bmu_n$, and the equivalence of Proposition~\ref{prop:non_abelian_reduction} recovers the orthogonal decomposition \eqref{eq:mun_gerbe_dec6}.
	\end{remark}

	\begin{example}[Total rigidification]
	\label{ex:total_rig}
		Let $X$ be a smooth Deligne--Mumford stack over a field $k$ of characteristic $0$. Then $X$ admits a \emph{total rigidification} \cite[Appendix A]{tame_stacks_AOC}, which is a gerbe $X \to X^{\rig}$ with finite inertia such that $X^{\rig}$ is an orbifold over $k$. (We recall our convention that an orbifold over $k$ is a Deligne--Mumford stack of finite type with trivial generic stabilizers). By Corollary~\ref{cor:non-abelian_gerbe_cover}, $\dperf(X)$ is equivalent to $\dperf(Y, \alpha)$, where $Y$ is an orbifold and $\alpha \in \Ho^2_{\et}(Y, \G_m)$.
	\end{example}

	\subsection{Root stacks}
	\label{ssec:root_stacks}

	In what follows, we have followed the notation and conventions of \cite[\S 3]{geometricity_dm}, to which we refer for details. 

	\begin{definition}
	\label{def:root_stack}
	    Let $X$ be an algebraic stack, and let $E \subset X$ be an effective Cartier divisor. For an integer $r > 0$, we recall that the \emph{$r$th root stack} $X_{r^{-1} E} \to X$ of $X$ along $E$ is given by the pullback 
			\[
				\begin{tikzcd}
					X_{r^{-1} E} \ar[d] \ar[r] & {[\A^1/\G_m]} \ar[d, "\pi_r"] \\
					X \ar[r, "f"] & {[\A^1/\G_m]},
				\end{tikzcd}
			\]
			where $f$ is the morphism induced by $E$, and $\pi_r$ is induced by $r$th power map on both $\A^1$ and $\G_m$.
	\end{definition}

	Let $\pi:X_{f^{-1} E} \to X$ be the $r$th stack of $X$ along $E$. According to \cite[Theorem 4.7]{geometricity_dm}, $\dperf(X_{r^{-1} E})$ admits an $X$-linear semiorthogonal decomposition of the form
	\begin{equation}
	\label{eq:root_stack_sod_dec6}
		\dperf(X_{r^{-1}E}) = \langle \scc_{r - 1}, \scc_{r - 2}, \dots, \scc_1,  \pi^* \dperf(X)\rangle,
	\end{equation}
	where each $\scc_i$ is equivalent to $\dperf(E)$. In fact, the base change of \eqref{eq:root_stack_sod_dec6} to the preimage $r^{-1} E$ of $E$ in $X_{r^{-1} E}$ recovers the semiorthogonal decomposition \eqref{eq:mun_gerbe_dec6} associated to the $\bmu_n$-gerbe $r^{-1} E \to E$, which has a trivial cohomological Brauer class.

		\begin{example}[Weak factorization]
		\label{ex:weak_factorization}
		According to a result of Harper \cite{harper} (see also \cite{bergh_weak_factor}), if $X$ and $Y$ are smooth, separated Deligne--Mumford stacks over a field of characteristic $0$, which are isomorphic over an open substack $U$, then $X$ and $Y$ may be connected by a chain 
		\[
			\begin{tikzcd}
				X = X_0 \ar[r, dashed, "f_0"] & X_1 \ar[r, dashed, "f_1"] & \cdots \ar[r, dashed, "f_{n - 1}"] & X_n = Y,
			\end{tikzcd}
		\]
		where, for each $i$, either $f_i$ or $f_i^{-1}$ is a root stack over a smooth divisor in the complement of $U$ or a blowup along a smooth closed substack in the complement of $U$.

		If $X_i \to X_i^{\rig}$ denotes the total rigidification of $X_i$ (Example~\ref{ex:total_rig}), then there is a diagram
		\[
			\begin{tikzcd}
				X = X_0 \ar[r, dashed, "f_0"] \ar[d] & X_1 \ar[d] \ar[r, dashed, "f_1"] & \cdots \ar[r, dashed, "f_{n - 1}"] & X_n \ar[d] = Y \\
				X^{\rig} = X_0^{\rig} \ar[r, dashed, "f_0^{\rig}"]  & X_1^{\rig} \ar[r, dashed, "f_1^{\rig}"] & \cdots \ar[r, dashed, "f_{n - 1}^{\rig}"] & X_n^{\rig}  = Y^{\rig},
			\end{tikzcd}
		\]
		where the vertical morphisms are gerbes with finite inertia.
	\end{example}

	\begin{variant}
		Suppose that $X$ is smooth over a field $k$, and let $E$ be a simple normal crossing divisor of $X$, with ordered components $E_1, \dots, E_n$. For a multi-index $r \in \Z_{> 0}^n$, the \emph{iterated $r$th root stack} $X_{r^{-1} E}$ is given by the fiber product 
		\[
			X_{r_1^{-1} E_1} \times_X X_{r_2^{-1} E_2} \times_X \cdots \times_X X_{r_n^{-1} E_n}.
		\]
		The stack $X_{r^{-1} E}$ is smooth over $k$, and the preimage of $E$ is a simple normal crossing divisor.  
	\end{variant}

	The following lemma is standard, but we include it for completeness.

	\begin{lemma}
	\label{lem:root_stack_brauer_class}
	    Let $X$ be a smooth variety over a field $k$, and let $U \subset X$ be an open subvariety such that $X - U$ has simple normal crossings with components $E_i$. If $X_{n^{-1} E} \to X$ is an iterated $\bar n$th root stack over $E = (E_i)$, where $\bar n = (n, n, \dots, n)$, then the restriction
	    \[
	    	\Br(X_{n^{-1}D})[n] \to \Br(U)[n]
	    \]
	    is an isomorphism. 
	\end{lemma}

	\begin{proof}
	    By purity, one may reduce to the case of a DVR, which is treated in \cite[\S 3.2]{lieblich_arithmetic_surface}.
	\end{proof}

	\begin{lemma}
	\label{lem:open_subvariety_gerbe}
	    Let $U$ be a smooth, quasi-projective variety over a field $k$ of characteristic $0$, and let $\alpha \in \Br(U)$. There exists a smooth, proper Deligne--Mumford stack $X$, and an open immersion $j:U \to X$, such that the following conditions hold:
	    \begin{enumerate} [label = (\arabic*)]
	    	\item There exists a class $\alpha' \in \Br(X)$ whose restriction to $U$ is $\alpha$.
	    	\item Any $\bmu_n$-gerbe $\scu \to U$ with Brauer class $\alpha$ extends to a $\bmu_n$-gerbe $\scx \to X$ of Brauer class $\alpha'$.
	    \end{enumerate}
	\end{lemma}

	\begin{proof}
	    Let $X_0$ be a smooth, projective variety over $k$ which compactifies $U$ such that $X_0 - U$ has simple normal crossings. We observe that $\Pic(X_0) \to \Pic(U)$ is surjective. By Lemma~\ref{lem:root_stack_brauer_class}, there is an iterated root stack $X \to X_0$ such that $\alpha$ extends to a cohomological Brauer class $\alpha' \in \Ho^2_{\et}(X, \G_m)[n]$. 

	    Let $\scu \to U$ be a $\bmu_n$-gerbe of Brauer class $\alpha$. Consider the commutative diagram
	    \[
	    	\begin{tikzcd}
	    		\Pic(X) \ar[r] \ar[d] & \Ho^2_{\et}(X, \bmu_n) \ar[r] \ar[d] & \Ho^2_{\et}(X, \G_m)[n] \ar[d] \\
	    		\Pic(U) \ar[r] & \Ho^2_{\et}(U, \bmu_n) \ar[r] & \Ho^2_{\et}(U, \G_m)[n].
	    	\end{tikzcd}
	    \]
	    The outer vertical morphisms are surjective, so $[\scu]$ lifts to a class $[\scx] \in \Ho^2_{\et}(X, \bmu_n)$. Finally, since $X$ is a global quotient stack with quasi-projective coarse space (\cite[Lemma 3.4]{geometricity_dm}),  the covering theorem of Kresch--Vistoli \cite{vistoli_kresch} implies that $\Ho^2_{\et}(X, \G_m)$ coincides with $\Br(X)$
	\end{proof}

\section{Hodge theory of categories} 
	\label{sec:hodge_theory_of_categories}

	Throughout this section, we work over the field $\C$ of complex numbers. In \S \ref{ssec:topological_k_theory}, we review Blanc's topological $K$-theory. In \ref{ssec:hodge_filtration}, we review Perry's construction of the Hodge structure on topological $K$-theory for geometric categories, and in \S \ref{ssec:variations_of_hodge_structure} we state the generalization to categories linear over a base.

	\subsection{Topological \textit{K}-theory}
	\label{ssec:topological_k_theory}

	In \cite{blanc}, Blanc constructs a functor 
	\[
		\ktop:\Cat_{\C} \longrightarrow \Sp
	\]
	from the $\infty$-category of $\C$-linear categories to the stable $\infty$-category of spectra, satisfying the following properties:

	\begin{enumerate} [label = \normalfont{(\roman*)}]
		\item \label{pt:colim} $\ktop$ commutes with filtered colimits.
		\item \label{pt:cofib} Given an exact sequence\footnote{A sequence of $\C$-linear categories is exact if it is both a fiber and a cofiber sequence.}
		\[
			\begin{tikzcd}
				\scc \ar[r] & \scc' \ar[r] & \scc''
			\end{tikzcd}
		\]
		of $\C$-linear categories, the resulting sequence
		\[
			\begin{tikzcd}
				\ktop(\scc) \ar[r] & \ktop(\scc') \ar[r] & \ktop(\scc'')
			\end{tikzcd}
		\]
		is an exact triangle in $\Sp$.

		\item \label{pt:chern_square} There is a commutative diagram
		\begin{equation} \label{eq:chern_square}
			\begin{tikzcd}
				\K(\scc) \ar[r, "\ch"] \ar[d] & \mathrm{HN}(\scc) \ar[d] \\
				\ktop(\scc) \ar[r, "\ch"] & \mathrm{HP}(\scc)
			\end{tikzcd}
		\end{equation}
		in $\Sp$, functorial in $\scc$.
		\item If $X$ is a separated scheme of finite type over $\C$, then $\ktop(\dperf(X))$ may be identified with the complex topological $K$-theory $\ktop_0(X)$, and \ref{pt:chern_square} may be identified (functorially in $X$) with the corresponding diagram
		\[
			\begin{tikzcd}
				\K(X) \ar[r, "\ch"] \ar[d] & \HN(X) \ar[d] \\
				\ktop(X^{\an}) \ar[r, "\ch"] & \HP(X)
			\end{tikzcd}
		\]
		of Chern characters for $X$.
	\end{enumerate}

	For our purposes, it is enough to understand $\ktop(\scc)$ in the case when $\scc$ occurs as a semiorthogonal component of a category of the form $\dperf(X)$, when $X$ is a scheme or a Deligne--Mumford stack. From \ref{pt:cofib}, the functor $\ktop$ satisfies the following additivity property: Given a semiorthogonal decomposition
	\[
	 	\scc = \langle \scc_1, \scc_2, \dots, \scc_n \rangle
	 \] 
	of a $\C$-linear category $\scc$, there is an equivalence
	\[
		\ktop(\scc) \cong \ktop(\scc_1) \oplus \ktop(\scc_2) \oplus \cdots \oplus \ktop(\scc_n)
	\]
	induced by the projections $\scc \to \scc_i$. 

	\begin{remark}
	\label{rem:euler_pairing}
		Let $\scc$ be a \emph{proper} $\C$-linear category. Perry \cite[Lemma 5.2]{Perry_CY2} constructs an \emph{Euler pairing} 
		\[
			\chi^{\topo}(-,-):\ktop_i(\scc) \otimes \ktop_i(\scc) \to \Z,
		\]
		which satisfies the following properties:
		\begin{enumerate} [label = (\arabic*)]
			\item If $\scc$ admits a semiorthogonal decomposition, then the resulting splitting of $\ktop_i(\scc)$ is semiorthogonal for the Euler pairing. 
			\item If $\chi(-, -)$ is the Euler pairing on $\K_0(\scc)$ given by
			\[
				\chi(v, w) = \sum_i (-1)^i \dim \Ext_{\C}^i(v, w),
			\]
			then the homomorphism $\K_0(\scc) \to \ktop_0(\scc)$ preserves Euler pairings. 
			\item If $\scc = \dperf(X)$ for a proper scheme $X$ over $\C$, then for $v, w \in \ktop_i(X)$,
			\[
				\chi^{\topo}(v, w) = s_*(v^{\vee} \otimes w) \in \ktop_{2i}(\Spec \C) \simeq \Z,
			\]
			where $s:X \to \Spec(\C)$ is the structure morphism. 
		\end{enumerate}
	\end{remark}

	Blanc's topological $K$-theory satisfies a localization sequence:

	\begin{lemma}
	\label{lem:excision_sequence}
	Let $i:Z \to X$ be a closed immersion of algebraic stacks over $\C$, with complement $j:U \to X$. There is an exact triangle
	\[
		\begin{tikzcd}
			\ktop(\dcoh(Z)) \ar[r, "i_*"] & \ktop(\dcoh(X)) \ar[r, "j^*"] & \ktop(\dcoh(X - Z)).
		\end{tikzcd}
	\]
	\end{lemma}

	The following argument appeared in a preprint version of \cite{equiv_hodge}, and we include it here for completeness. Note that it exploits details of Blanc's construction (cf. \cite[Def. 1.2]{blanc}), which we have not described.

	\begin{proof}
		For a $\C$-linear category $\scc$, consider the presheaf of spectra 
		\[
			\sK(\scc)(U) = \sK(\scc \otimes_{\C} \dperf(U))
		\]
		on the category of smooth, affine schemes over $\C$. 

		For any smooth, affine scheme $U$, and any algebraic stack $X$, there is an equivalence
		\[
			\dcoh(X) \otimes_{\C} \dperf(U) \cong \dcoh(X \times_{\C} U)
		\]
		by \cite[Corollary 4.2.3]{drinfeld_gaitsgory}. From the localization sequence for $G$-theory of algebraic stacks, it follows that there is a fiber sequence of presheaves of spectra
		\begin{equation} \label{eq:localization_alg}
			\begin{tikzcd}
				\sK(\dcoh(Z)) \ar[r] & \sK(\dcoh(X)) \ar[r] & \sK(\dcoh(U))
			\end{tikzcd}
		\end{equation}
		on the category of smooth, affine schemes over $\C$. Finally, \eqref{eq:localization_alg} remains a fiber sequence after geometric realization and inverting the Bott element. 
	\end{proof}

	\begin{example}
		Let $X$ be a smooth, projective Deligne--Mumford stack over $\C$, and suppose that $X = [Y/G]$, where $Y$ is quasi-projective and $G$ is reductive. According to a result of Halpern-Leistner and Pomerleano, one may identify $\ktop(\dperf(X))$ with the equivariant topological $K$-theory $\ktop_G(Y^{\an})$ \cite[Theorem 2.10]{equiv_hodge}.
	\end{example}

	\begin{example}
		Let $X$ be a separated scheme of finite type over $\C$, and $\alpha \in \Br(X)$. Moulinos \cite{moulinos} constructs an equivalence 
		\[
			\ktop(\dperf(X, \alpha)) \cong \KU^{\bar \alpha}(X^{\an}),
		\]
		where $\bar \alpha \in \Ho^3(X^{\an}, \Z)^{\tors}$ is the topological Brauer class associated to $\alpha$, and $\KU^{\bar \alpha}(X^{\an})$ is the $\bar \alpha$-twisted topological $K$-theory spectrum of $X$ in the sense of Atiyah--Segal \cite{atiyah_segal}.
	\end{example}

	\begin{remark} \label{rem:descent_for_ktop}
		Let $\scc$ be a $\C$-linear category. According to a result of Antieau--Heller \cite{antieau_heller}, the functor
		\[
			T \mapsto \ktop(\dperf(T) \otimes_{\C} \scc)
		\]
		satisfies \'etale hyperdescent on the site of smooth, separated schemes over $\C$. In the case $\scc = \dperf(\Spec \C)$, one recovers the fact that topological $K$-theory satisfies \'etale hyperdescent. 
	\end{remark}

	\begin{variant} \label{relative_moulinos}
	There is a relative version of Blanc's construction, due to Moulinos \cite{moulinos}. Let $S$ be a separated scheme of finite type over $\C$. There is a functor
	\begin{equation} \label{eq:relative_blanc}
		\ktop(-/S):\Cat_S \longrightarrow \Shv_{S^{\an}}(\Sp),
	\end{equation}
	valued in sheaves of spectra on $S^{\an}$, satisfying the following properties:
	\begin{enumerate} [label = \normalfont{(\roman*)}]
		\item When $S = \Spec \C$, $\ktop(\scc/S)$ is the constant sheaf associated to $\ktop(\scc)$.
	\item If $\scc$ admits an $S$-linear semiorthogonal decomposition 
	\[
		\scc = \langle \scc_1, \scc_2, \dots, \scc_n \rangle,
	\]
	then there is a decomposition
	\[
		\ktop(\scc/S) = \ktop(\scc_1/S) \oplus \ktop(\scc_2/S) \oplus \cdots \oplus \ktop(\scc_n/S).
	\]
	\item If $f:X \to S$ is a proper morphism of separated schemes over $\C$, then $\ktop(\dperf(X)/S)$ may be identified (functorially in $S$) with the sheaf of spectra $U \mapsto \ktop((f^{\an})^{-1}(U))$.
	\end{enumerate}
	\end{variant}

	\subsection{The Hodge filtration} 
	\label{ssec:hodge_filtration}

	Let $\scc$ be a $\C$-linear category. In the terminology of Orlov~\cite{orlov_gluing}, $\scc$ is \emph{geometric} if it arises as a semiorthogonal component of $\dperf(X)$, for a smooth, proper scheme $X$ over $\C$. Examples of $\scc$-linear categories include $\dperf(X)$ for smooth, proper Deligne--Mumford stacks \cite{geometricity_dm}, and, consequently, $\dperf(X, \alpha)$ where $X$ is a smooth, proper Deligne--Mumford stack, and $\alpha \in \Ho^2(X, \G_m)$.

	\begin{construction} \label{constr:nc_hodge}
		Let $\scc$ be a geometric $\C$-linear category. For each $i \in \Z$, the morphism from \ref{pt:chern_square}
		\[
			\ktop_i(\scc) \otimes \C \longrightarrow \HP_i(\scc)
		\]
		is an isomorphism, the noncommutative Hodge-to-de Rham spectral sequence
		\begin{equation} \label{eq:nc_hodge_ss}
			\HH_*(\scc)[u^{\pm 1}] \implies \HP_*(\scc)
		\end{equation}
		degenerates \cite{kaledin}, and the resulting filtration endows $\ktop_i(\scc)$ with a pure Hodge structure of weight $-i$. 
	\end{construction}

	\begin{example}
		Let $X$ be a smooth, proper scheme over $\C$, and let $i \in \Z$. Identifying $\ktop_i(X)$ with $\ktop_i(\dperf(X))$, the result of Construction~\ref{constr:nc_hodge} is the unique integral Hodge structure on $\ktop_i(X)$ such that the Chern character homomorphism
		\[
			\ktop_i(X) \otimes \Q \to \bigoplus_{k} \Ho^{2k - i}(X^{\an}, \Q)(k)
		\]
		is an isomorphism of rational Hodge structures.
	\end{example}

	\begin{remark}
		If $X$ is a smooth, proper Deligne--Mumford stack over $\C$, then the Hodge structure on $\ktop_i(\dperf(X))$ is not necessarily pulled back from the Hodge structures on the rational cohomology of $X$ through a Chern character. For example, when $X = \B G$ for a finite cyclic group $G$, then $\ktop_0(\dperf(\B G))$ has rank $n$, whereas $\Ho^{2*}(X^{\an}, \Q)$ has rank $1$.
	\end{remark}

	For a geometric category $\scc$, we write $\Hdg(\scc, \Z)$ for the group of integral Hodge classes in $\ktop_0(\scc)$. From considering the case of $\dperf(X)$, it follows that there is a factorization
	\[
		\K_0(\scc) \longrightarrow \Hdg(\scc, \Z) \subset \ktop_0(\scc).
	\]
	Perry formulates a noncommutative integral Hodge conjecture \cite[Conjecture 5.13]{Perry_CY2}:

	\begin{conjecture}[Perry]
		\label{conj:ihc}
	    Let $\scc$ be a geometric $\C$-linear category. The homomorphism
	    \[
	    	\K_0(\scc) \to \Hdg(\scc, \Z)
	    \]
	    is surjective. 
	\end{conjecture}

	As with the integral Hodge conjecture for the integral cohomology of a smooth, proper variety, Conjecture~\ref{conj:ihc} is known to fail in certian examples. As Perry notes, if $X$ is a hypersurface in $\P^4$ for which the integral Hodge conjecture fails \cite{trento}, then Conjecture~\ref{conj:ihc} fails for $\dperf(X)$.

	\begin{remark} \label{rem:implications_ihc}
		If the Atiyah--Hirzebruch spectral sequence relating $\ktop_*(X)$ with $\Ho^*(X^{\an}, \Z)$ degenerates (for example, if the cohomology ring $\Ho^*(X^{\an}, \Z)$ is torsion-free) then the integral Hodge conjecture in all degrees for $X$ implies the integral Hodge conjecture for $\Dperf(X)$, as one may see from the proof of \cite[Prop. 5.16]{Perry_CY2}. In general, however, there do not appear to be implications in either direction between the integral Hodge conjecture for $X$ and the integral Hodge conjecture for $\dperf(X)$.
	\end{remark}
	
	It is often useful to measure the failure of the integral Hodge conjecture using the \emph{Voisin group},
	\[
		\rV(\scc) = \Coker \left( \K_0(\scc) \to \Hdg(\scc, \Z) \right).
	\]
	The integral Hodge conjecture holds for $\scc$ if and only if $\rV(\scc) = 0$. 

	The most important property of the integral Hodge conjecture for categories is its compatibility with semiorthogonal decompositions: If $\scc$ admits a semiorthogonal decomposition
	\[
		\scc = \langle \scc_1, \scc_2, \dots, \scc_n \rangle,
	\]
	then there is a decomposition
	\[
		\rV(\scc) = \rV(\scc_1) \oplus \rV(\scc_2) \oplus \cdots \oplus \rV(\scc_n).
	\]
	In particular, the integral Hodge conjecture holds for $\scc$ if and only if it holds for each $\scc_i$.

	\subsection{Variations of Hodge structure}
	\label{ssec:variations_of_hodge_structure}

	Let $S$ be a separated scheme of finite type over $\C$, and $f: X \to S$ be a smooth, proper morphism. Suppose that there is an $S$-linear semiorthogonal decomposition 
	\[
		\dperf(X) = \langle \scc_1, \scc_2, \dots, \scc_n \rangle.
	\]
	For each $i \in \Z$, there is a splitting 
	\begin{equation} \label{eq:split_VHS}
		\ktop_i(X/S) = \ktop_i(\scc_1/S) \oplus \ktop_i(\scc_2/S) \oplus \cdots \oplus \ktop_i(\scc_n/S),	
	\end{equation}
	and combining Construction \ref{constr:nc_hodge} with Variant~\ref{relative_moulinos}, one may show that \eqref{eq:split_VHS} is a splitting of pure variations of Hodge structure.

	\section{Twisted Mukai structures}
	\label{sec:twisted_mukai_structures}

	In \S \ref{ssec:topologically_trivial_brauer_classes}, we review the notion of a topologically trivial Brauer class. In \S \ref{ssec:severi_brauer_varieties}, we prove the key input (Proposition~\ref{prop:key_input}) on the topological $K$-theory of Severi--Brauer varieties. In \S \ref{ssec:twisted_mukai}, we prove the main result on twisted Mukai structures for smooth, proper varieties. In \S \ref{ssec:variation_of_twisted_mukai}, we prove a generalization to variations of twisted Mukai structures over a base. In \S \ref{ssec:quasi_proj}, we develop a version of twisted Mukai structures for smooth, quasi-projective varieties. In \S \ref{ssec:computations_on_quasi_proj}, we prove auxiliary results on twisted Mukai structures on quasi-projective varieties, which will be used later.

	\subsection{Topologically trivial Brauer classes} 
	\label{ssec:topologically_trivial_brauer_classes}

	For an integer $n$, we write $\Z(n) = (2 \pi i)^n \Z$.

	\begin{definition}
	\label{def:sep_scheme_finite_type}
	    Let $X$ be a separated scheme of finite type over $\C$, and let $\alpha \in \Ho^2_{\et}(X, \G_m)^{\tors}$ be a cohomological Brauer class. Then $\alpha$ is \emph{topologically trivial} if it lies in the kernel of the composition
	    \[
	    	\Ho^2_{\et}(X, \G_m) \to \Ho^2(X^{\an}, \G_m^{\cont}) \to \Ho^3(X^{\an}, \Z(1)),
	    \]
	    where $\G^{\cont}_m$ is the sheaf of invertible continuous functions on $X^{\an}$, and the right-hand morphism arises from the continuous exponential sequence.
	\end{definition}

	\begin{remark}
		If $X$ is proper, then $\Ho^2_{\et}(X, \G_m)^{\tors}$ is isomorphic to $\Ho^2(X^{\an}, \sco_{X^{\an}}^\times)^{\tors}$. From the exponential sequence for the complex-analytic space $X^{\an}$, there is an exact sequence
		\[
			\begin{tikzcd}
				\Ho^2(X^{\an}, \Q(1)) \ar[r] & \Ho^2_{\et}(X, \G_m)^{\tors} \ar[r] & \Ho^3(X^{\an}, \Z(1))^{\tors},
			\end{tikzcd}
		\]
		so $\alpha \in \Ho^2_{\et}(X, \G_m)^{\tors}$ is topologically trivial if and only if $\alpha$ lies in the image of $\Ho^2(X^{\an}, \Q(1))$.
	\end{remark}

	It is often convenient to phrase topological triviality at the level of $\bmu_n$-gerbes:

	\begin{definition}
	\label{def:essentially_topologically_trivial}
	    Let $X$ be a separated scheme of finite type over $\C$, and let $\scx \to X$ be a $\bmu_n$-gerbe, for $n > 0$. 
	    \begin{enumerate} [label = (\arabic*)]
	    	\item  If the $\G_m$-gerbe associated to $\scx$ is trivial, then $\scx$ is \emph{essentially trivial}.
	    	\item If the topological $\G_m$-gerbe associated to $\scx^{\an}$ is trivial, then $\scx$ is \emph{essentially topologically trivial}. 
	    \end{enumerate}
	\end{definition}

 	We observe that $\scx$ is essentially trivial if and only if $[\scx]$ lies in the kernel of the homomorphism 
	\[
	\Ho^2_{\et}(X, \bmu_n) \to \Ho^2_{\et}(X, \G_m).
	\] 
	Similarly, $\scx$ is essentially topologically trivial if and only if $[\scx]$ lies in the kernel of the connecting homomorphism
	    \[
	    	\Ho^2(X^{\an}, \bmu_n) \to \Ho^3(X^{\an}, \Z(1))
	    \]
	arising from the sequence 
	    \[
	    	\begin{tikzcd}[column sep=large]
	    		\Z(1) \ar[r, "n"] & \Z(1) \ar[r, "\exp(-/n)"] & \bmu_n.
	    	\end{tikzcd}
	    \]
	Equivalently, $[\scx]$ lies in the image of an element $v \in \Ho^2(X^{\an}, \Z(1))$. 

	\begin{remark}
		Let $\scx \to X$ be a $\bmu_n$-gerbe. By a \emph{topological line bundle} on $\scx$, we mean a line bundle on the associated gerbe on the topological site of $X$. If $L$ is a topological line bundle on $\scx$, then $L^{\otimes n}$ descends to a topological line bundle on $X$, which we simply call $L^{\otimes n}$. 

		Similarly, we will consider $1$-twisted topological line bundles on $\scx$. If $\scx \to X$ is topologically essentially trivial, the set of $1$-twisted topological line bundles on $\scx$ is a torsor under the group of topological line bundles on $X$ 
	\end{remark}

	\begin{definition}
	\label{def:chern_character}
	    Let $X$ be a separated scheme of finite type over $\C$, and let $\scx \to X$ be a $\bmu_n$-gerbe. Let $L$ be a topological line bundle on $\scx$. 
	    \begin{enumerate} [label = (\arabic*)]
	    	\item The $K$-theory class $[L]$ of $L$ in $\ktop_0(X) \otimes \Q$ is the $n$th root $([L^{\otimes}])^{1/n}$ of the class of $L^{\otimes n}$.
	    	\item The \emph{first Chern class} of $L$ is $c_1(L) = c_1(L^{\otimes n})/n$, regarded as an element of $\Ho^2(X^{\an}, \Q(1))$.
	    	\item The \emph{Chern character} of $L$ is the image of $[L]$ under the Chern character:
		    \[
		    	\ch(L) = \exp(c_1(L)) \in \Ho^{2*}(X^{\an}, \Q(*)).
		    \]
	    \end{enumerate}	 
	\end{definition}

	\begin{lemma}
	\label{lem:exp_of_c_1}
	    Let $L$ be a $1$-twisted topological line bundle on $\scx$. The image of $c_1(L^{\otimes n})$ under 
	    \[
	    	\exp(-/n):\Ho^2(X^{\an}, \Z(1)) \to \Ho^2(X^{\an}, \bmu_n)
	    \]
	    is $[\scx]$.
	\end{lemma}

	\begin{proof}
		This is an analogue of \cite[Prop. 2.3.4.4]{lieblich_moduli} for the topological site, and the proof is identical. Briefly, if $L$ is a $1$-twisted topological line bundle on $\scx$, then the topological gerbe associated to $\scx$ is equivalent to the gerbe of $n$th roots of $L^{\otimes n}$.
	\end{proof}

	\subsection{Severi--Brauer varieties}
	\label{ssec:severi_brauer_varieties}

	Let $X$ be a separated scheme of finite type over $\C$. Let $\scx \to X$ be a $\bmu_n$-gerbe, and suppose that there exists a Severi--Brauer variety $\P \to X$ which represents the image of $[\scx]$ in $\Ho^2_{\et}(X, \G_m)$. In what follows, we frequently refer to the notation of Example~\ref{ex:bernardara}.

	Consider the pullback diagram
	\[
		\begin{tikzcd}
			\P_{\scx} \ar[d, "\pi'"] \ar[r, "f'"] & \P \ar[d, "\pi"] \\
			\scx \ar[r, "f"] & X.
		\end{tikzcd}
	\]
	Let $L$ be a $1$-twisted topological line bundle on $\scx$. Then $\pi'^* L \otimes \sco_{\P_{\scx}}(1)$	is a $0$-twisted topological line bundle on $\P_{\scx}$, so it descends to a line bundle on $\P$, which we simply call $L(1)$.

	\begin{proposition}
	 \label{prop:key_input}
	     In the setting above, multiplication by $L(1)$ on $\ktop_0(\P)$ induces an equivalence of spectra
	     \[
	     	\ktop(\dperf(\P)_0) \to \ktop(\dperf(\P)_1).
	     \]
	 \end{proposition} 

	 \begin{proof}
	     First, assume that $\scx \to X$ is essentially trivial. We may write $L = L^{\alg} \otimes f^*L_0$, where $L^{\alg}$ is a $1$-twisted algebraic line bundle on $\scx$ and $L_0$ is a topological line bundle on $X$. Then $L(1) = L^{\alg}(1) \otimes \pi^* L_0$, and $L^{\alg}(1)$ is a tautological bundle on the projective bundle $\P \to X$, so the desired result is clear from Beilinson's semiorthogonal decomposition. 

	     In the general case, consider the pullback square
	     \[
	     	\begin{tikzcd}
	     		\P^{(2)} = \P \times_X \P \ar[r, "\pi_2"] \ar[d, "\pi_1"] & \P \ar[d] \\
	     		\P \ar[r] & X.
	     	\end{tikzcd}
	     \]
	     Let $\dperf(\P^{(2)})^{\pi_1}_k$ be the $k$th piece of the semiorthogonal decomposition of $\dperf(\P^{(2)})$ for the projective bundle structure from $\pi_1$. The bottom horizontal morphism of the commutative diagram
	     \[
	     	\begin{tikzcd}
	     		\ktop(\dperf(\P)_0) \ar[d, "\pi_2^*"] \ar[r, "L(1)"] & \ktop(\dperf(\P)_1) \ar[d, "\pi_2^*"] \\
	     		\ktop(\dperf(\P^{(2)})^{\pi_1}_0) \ar[r, "\pi_2^* L(1)"] & \ktop(\dperf(\P^{(2)})^{\pi_1}_1).
	     	\end{tikzcd}
	     \]
	     is an equivalence by the previous case, and moreover the vertical morphisms are inclusions of summands. Comparing with the analogous diagram for multiplication by $L(1)^{-1}$, we see that multiplication by $L(1)$ is an equivalence. 
	 \end{proof}

	 \begin{remark}
	 	The assumption in Proposition~\ref{prop:key_input} that the Brauer class of $\scx$ is represented by a Severi--Brauer variety is, we expect, inessential. However, we have imposed it in the absence of a comparison theorem between the topological $K$-theory of the gerbe $\scx$ and Blanc's $K$-theory in the case when $\scx$ is not a global quotient. 
	 \end{remark}

	\subsection{Twisted Mukai structures}
	\label{ssec:twisted_mukai}

	The goal of this section is to prove that the Hodge structures of twisted derived categories are given by twisted Mukai structures. This is stated as Corollary~\ref{cor:twisted_mukai-b_field} below.

	\begin{definition}
	\label{def:twisted}
	    Let $X$ be a smooth, proper variety over $\C$. Given a class $B \in \Ho^2(X^{\an}, \Q(1))$, the \emph{$B$-twisted Mukai structure} $\ktop_i(X)^B$ is the unique integral Hodge structure of weight $-i$ on $\ktop_i(X)$ such that the homomorphism
	    \[
	    	\ktop_i(X)^B \to \bigoplus_k \Ho^{2k - i}(X^{\an}, \Q(k)), \hspace{5mm} v \mapsto \exp(B) \cdot \ch(v)
	    \]
	    induces, after extension of scalars, an isomorphism between rational Hodge structures. We endow $\ktop_i(X)^B$ with an \emph{Euler pairing}, which by definition is the Euler pairing of Remark~\ref{rem:euler_pairing} on the underlying group $\ktop_i(X)$.
	\end{definition}

	\begin{theorem}
	\label{thm:ktop_of_gerbe}
	    Let $X$ be a smooth, proper variety over $\C$, and let $\scx \to X$ be an essentially topologically trivial $\bmu_n$-gerbe. Suppose that the cohomological Brauer class of $\scx$ lies in $\Br(X)$. Then for each $i$ and each $1$-twisted topological line bundle $L$ on $\scx$, there is an isomorphism of Hodge structures
	    \[
	    	\ktop_i(\dperf^1(\scx)) \to \ktop_i(X)^{c_1(L)},
	    \]
	    which is compatible with Euler pairings. 
	\end{theorem}

	\begin{proof}
			Let $\P$ be a Severi--Brauer variety which represents the image of $[\scx]$ in $\Br(X)$. In the notation of Proposition~\ref{prop:key_input}, consider the composition
			\[
				\begin{tikzcd}[column sep=large]
					\ktop_i(\dperf(\P)_0) \ar[r, "{[L(1)]}"] & \ktop_i(\dperf(\P)_1) \ar[r, "{[\sco_{\P_{\scx}}(-1)]}"] & \ktop_i(\dperf(\P)_0) \otimes \Q,
				\end{tikzcd}
			\]
			where, in the right-hand morphism, we use multiplication by the $K$-theory class of the twisted line bundle (Definition~\ref{def:chern_character}). 

			Since $\sco_{\P_{\scx}}(-1)$ is an algebraic twisted line bundle, the right-hand morphism induces an isomorphism of rational Hodge structures. On the other hand, by Proposition~\ref{prop:key_input}, the left-hand morphism is an isomorphism of abelian groups. It follows that the Hodge structure on $\ktop_i(\dperf(\P)_1))$ is identified with $\ktop_i(X)^{c_1(L)}$ by $[L(1)]$, and we conclude the theorem by Lemma~\ref{lem:weighted_lb_lemma}, which implies that $\ktop_i(\dperf(\P)_1)$ is isomorphic to $\ktop_i(\dperf^1(\scx))$ as a Hodge structure. The compatibility with Euler pairings follows from the observation that multiplication by $[L(1)]$ preserves the Euler pairing on $\ktop_i(\P)$.
	\end{proof}

	\begin{remark}
		We frequently use the observation that, in the case $i = 0$, the isomorphism constructed in the course of the proof preserves the rank homomorphisms
		\[
			\ktop_0(\dperf^1(\scx)) \to \Z, \hspace{5mm}  \ktop_0(X)^{c_1(L)} \to \Z
		\]
		on both sides. 
	\end{remark}

	\begin{definition}
	\label{def:rational_b_field}
	    Let $X$ be a smooth, proper variety over $\C$, and let $\alpha \in \Br(X)$. A \emph{rational $B$-field} for $\alpha$ is a class $B \in \Ho^2(X^{\an}, \Q(1))$ which maps to $\alpha$ under the exponential
	    \[
	    	\exp: \Ho^2(X^{\an}, \Q(1)) \to \Ho^2(X^{\an}, \sco_{X^{\an}}^\times)^{\tors} \simeq \Ho^2_{\et}(X, \G_m).
	    \]
	\end{definition}

	\begin{remark}
	\label{rem:rational_b_field}
		Let $\scx \to X$ be a $\bmu_n$-gerbe with cohomological Brauer class $\alpha$. If $L$ is a $1$-twisted topological line bundle on $\scx$, then $c_1(L^{\otimes n})/n$ is a rational $B$-field for $\alpha$. To prove this, observe that there is a commmutative diagram
		\[
			\begin{tikzcd}[column sep=large]
			 	\Ho^2(X^{\an}, \Z(1)) \ar[d, "-/n"] \ar[r, "\exp(-/n)"] &	\Ho^2(X^{\an}, \bmu_n) \ar[d] \\
			 	\Ho^2(X^{\an}, \Q(1)) \ar[r, "\exp"] & \Ho^2(X^{\an}, \sco_{X^{\an}}^{\times}),
			\end{tikzcd}
		\]
		and apply Lemma~\ref{lem:exp_of_c_1}. 
	\end{remark}

	\begin{corollary}	
	\label{cor:twisted_mukai-b_field}
	    Let $X$ be a smooth, proper variety over $\C$, and let $\alpha \in \Br(X)$ be a topologically trivial Brauer class. If $B$ is a rational $B$-field for $\alpha$, then for each $i$ there is an isomorphism of Hodge structures
	    \[
	    	\ktop_i(\dperf(X, \alpha)) \simeq \ktop_i(X)^B,
	    \]
	    which is compatible with Euler pairings. 
	\end{corollary}

	\begin{proof}
			By Remark~\ref{rem:rational_b_field} and Theorem~\ref{thm:ktop_of_gerbe}, it is enough to show that the isomorphism class of the integral Hodge structure $\ktop_i(X)^B$ does not depend on the choice of the rational $B$-field $B$. If $B$ and $B'$ are two rational $B$-fields for $\alpha$, then from the exponential sequence, the difference $B - B'$ lies in $\Ho^2(X^{\an}, \Z(1)) + \NS(X)_{\Q}$. 

			If $B - B'$ lies in $\NS(X)_{\Q}$, then the identity map on $\ktop_i(X)$ induces an isomorphism of Hodge structures 
			\[
			 	\ktop_i(X)^B \to \ktop_i(X)^{B'}.
			 \] 
			Therefore, we may suppose that $x = B - B'$ lies in $\Ho^2(X^{\an}, \Z(1))$. In that case, the desired isomorphism of Hodge structures
			\[
				\ktop_i(X)^B \to \ktop_i(X)^{B'}
			\]
			is given by multiplication by $[L_x]$, where $L_x$ is a topological line bundle on $X$ with $c_1(L_x) = x$.
	\end{proof}

	\subsection{Variation of twisted Mukai structure}
	\label{ssec:variation_of_twisted_mukai}

	Let $f:X \to S$ be a smooth, proper morphism, and let $\scx \to X$ be a $\bmu_n$-gerbe. Let $\rP^0$ be the local system of topological line bundles on the fibers $X_s$ of $f$, and let $\rP^1$ be the local system of $1$-twisted topological line bundles on the fibers of $\scx \to S$. We observe that $\rP^1$ is a torsor under $\rP^0$.

	\begin{remark}
		The category of $\rP^0$-sheaves is symmetric monoidal, with the monoidal structure given by the \emph{contracted product} of $\rP^0$-sheaves $A \wedge B$, which is the sheafification of the quotient presheaf
	 \[
	 	A \times B/(a, b) \sim (ga, g^{-1} b).
	 \]
	 If $A$ is a $\rP^0$-module, and $B$ is a $\rP^0$-torsor, then $A \wedge B$ is a $\rP^0$-module, with addition given locally by
	 \[
		\begin{split}
		a \wedge b + a' \wedge b' &= a \wedge b + a' \wedge gb \\
			&= (a + ga') \wedge b,
	\end{split}
	\]
	where $b' = gb$ for a unique $g \in \rP^0$. 
	\end{remark}

	For $i \in \Z$, the local system $\ktop_i(X/S)$ is a $\rP^0$-module, and we may define a variation of Hodge structure on the $\rP^0$-module $\ktop_i(X/S) \wedge \rP^1$ as follows:

	\begin{definition}
	\label{def:variation_of_mukai_structure}
	    The \emph{variation of twisted Mukai structure} on $\ktop_i(X/S) \wedge \rP^1$ is the unique integral variation of Hodge structure such that the Chern character 
	    \[
	    	\ktop_i(X/S) \wedge \rP^1 \to \bigoplus_k \Ri^{2k -i} f_* \Q(k), \hspace{5mm} v \wedge L \mapsto \ch(L) \cdot \ch(v)
	    \]
	    induces an isomorphism between $\Q$-variations of Hodge structure.
	\end{definition}

	\begin{theorem}
	\label{thm:twisted_mukai_in_families}
	    Let $f:X \to S$ be a smooth, proper morphism, where $S$ is a separated scheme of finite type over $\C$, and let $\scx \to X$ be a $\bmu_n$-gerbe such that the fiber $\scx_s \to X_s$ is essentially topologically trivial for each $s \in S(\C)$. Assume that the cohomological Brauer class of $\scx$ lies in $\Br(X)$. Then there is an isomorphism of variations of Hodge structure
	    \[
	    	\ktop_i(\dperf^1(\scx)/S)  \to \ktop_i(X/S) \wedge \rP^1.
	    \]
	\end{theorem}

	\begin{proof}
	    Let $\P \to X$ be a Severi--Brauer variety which represents the image of $[\scx]$ in $\Br(X)$. Following the proof of Theorem~\ref{thm:ktop_of_gerbe}, we may consider a composition of morphisms of $\rP^0$-modules
	    \begin{equation} 
	    \label{eq:composition_relative_case}
	    	\begin{tikzcd}[column sep=large]
	    		\ktop_i(\dperf(\P)_0/S) \wedge \rP^1 \ar[r, "\Phi"] & \ktop_i(\dperf(\P)_1/S) \ar[r, "\Psi"] & \ktop_i(\dperf(\P)_0/S) \otimes \Q,
	    	\end{tikzcd}
	    \end{equation}
	    where:
	    \[
	    	\Phi(v \wedge L) = [L(1)] \cdot v, \hspace{5mm} \Psi(v) = [\sco_{\P_{\scx}}(-1)] \cdot v
	    \]
	    On stalks, \eqref{eq:composition_relative_case} recovers (after choosing a $1$-twisted topological line bundle) the analogous sequence in the proof of Theorem~\ref{thm:ktop_of_gerbe}. In particular, $\Phi$ is an isomorphism of local systems, and $\Psi$ induces an isomorphism of $\Q$-variations of Hodge structure.
	\end{proof}

	\begin{remark}
		In the context of Theorem~\ref{thm:twisted_mukai_in_families}, let $\V = \ktop_0(X/S) \wedge \rP^1$. Then $\V$ admits a rational section $\tau \in \Gamma(S^{\an}, \V) \otimes \Q$, whose restriction to the stalk at $s \in S(\C)$ is given by $[L^{\vee}] \wedge L$ for any $1$-twisted topological line bundle $L$ on $\scx_s$. Then $\tau$ is Hodge, since $\Psi(\Phi(\tau)) = [L] \cdot [L^{\vee}] = 1$. 

		By the theorem of the fixed part \cite[4.1.1]{deligne_hodge_ii}, $\Gamma(\V) = \Gamma(S^{\an}, \V)$ carries a natural Hodge structure, determined by the property that for each $s \in S(\C)$, the inclusion $\Gamma(\V) \to \V_s$ is a morphism of Hodge structures. Regarding $\Gamma(\V)$ as a $\ktop_0(X/S)$-module, the section $\tau \in \Gamma(\V) \otimes \Q$ determines an isomorphism of $\Q$-variations of Hodge structure
		\[
			\ktop_0(X/S) \otimes \Q \to \V \otimes \Q,
		\]
		and hence an isomorphism of rational Hodge structures 
		\[
			\Gamma(\ktop_0(X/S)) \otimes \Q \simeq \Gamma(\V) \otimes \Q.
		\]
		In particular, if $\mathbf{v} \in \Gamma(\V)$ is everywhere Hodge, then the image of $\mathbf{v}$ in $\Gamma(\V) \otimes \Q$ may be written $\mathbf{w} \cdot \tau$, where $\mathbf{w}$ is a global section of $\ktop_0(X/S) \otimes \Q$ which is everywhere Hodge.
	\end{remark}

	\subsection{Quasi-projective varieties}
	\label{ssec:quasi_proj}

	We now develop a theory of twisted Mukai structures on smooth, quasi-projective varieties. On the one hand, our treatment is rather ad hoc, but on the other, it is amenable to computation. The usefulness of the quasi-projective case comes from the following well-known lemma:

	\begin{lemma}
	\label{lem:bloch_kato}
	    Let $X$ be a smooth scheme over $\C$, and let $\alpha \in \Ho^2_{\et}(X, \G_m)$ be a cohomological Brauer class. There is a dense open subscheme $U$ of $X$ such that $\alpha_U$ is topologically trivial.
	\end{lemma}

	\begin{proof}
			Let $\bar \alpha \in \Ho^3(X^{\an}, \Z(1))^{\tors}$ be the topological Brauer class of $\alpha$. According to \cite[Th\'eor\`eme 3.1]{CT_Voisin},  $\sch^3(\Z(1))$ is torsion-free, where $\sch^3(\Z(1))$ is the sheafification of the Zariski presheaf $U \mapsto \Ho^3(U^{\an}, \Z(1))$. In particular, $\bar \alpha$ lies in the kernel of the homomorphism
			\[
				\Ho^3(X^{\an}, \Z(1)) \to \Ho^0(X, \sch^3(\Z(1))).
			\]
			The right-hand side is the third unramified cohomology group $\Ho^3_{\mathrm{nr}}(X, \Z(1))$ of $X$, and the kernel consists of classes with coniveau $\leq 1$.
	\end{proof}

	Let $U$ be a smooth, separated variety over $\C$. Each singular cohomology group $\Ho^k(U^{\an}, \Q)$ carries a functorial mixed Hodge structure, with weights in the interval $[k, 2k]$. The lowest-weight part $\Wi_k \Ho^k(U^{\an}, \Q)$ is a pure Hodge structure of weight $k$, and may be described concretely as the image of the restriction map
	\[
		j^*:\Ho^k(X^{\an}, \Q) \to \Ho^k(U^{\an}, \Q)
	\]
	for any smooth compactification $j:U \to X$ \cite[3.2.16]{deligne_hodge_ii}.

	\begin{definition}
	\label{def:open_twisted_mukai}
	    Let $U$ be a smooth, separated scheme over $\C$, and let $B \in \Ho^2(U^{\an}, \Q(1))$. The \emph{$B$-twisted Mukai structure} $\wtop_i(U)^B$ is the Hodge structure of weight $-i$, given by the preimage of 
	    \begin{equation} \label{eq:lowest_weight_sum_twist}
	    	\bigoplus_k \Wi_{-i}(\Ho^{2k - i}(U^{\an}, \Q(k)) 
	    \end{equation}
	    under the $B$-twisted Chern character
	    \[
	    	\ktop_i(U) \longrightarrow \bigoplus_k \Ho^{2k - i}(U^{\an}, \Q(k)), \hspace{5mm} v \mapsto \exp(B)\ch(v).
	    \]
	    The Hodge filtration on $\wtop_i(U)^B$ is pulled back along the Chern character from \eqref{eq:lowest_weight_sum_twist}. When $B = 0$, we write $\wtop_i(U)^B = \wtop_i(U)$, and call $\wtop_i(U)$ the \emph{lowest-weight part} of $\ktop_i(U)$.
	\end{definition}

		\begin{lemma}
	\label{lem:lowest_weight_does_not_depend}
	    Let $U \to X$ be an open immersion from a smooth, separated scheme $U$ to a smooth, proper Deligne--Mumford stack $X$ over $\C$. For each $i$, the homomorphism $\ktop_i(\dperf(X))$ to $\ktop(U)$ factors through $\wtop_i(U)$, and the induced homomorphism
	    \[
	    	\ktop_i(\dperf(X)) \to \wtop_i(U)
	    \]
	    is a morphism of Hodge structures.
	\end{lemma}

	\begin{proof}
		First, suppose that $X$ is a smooth, proper variety. Then the lemma follows from the fact that for each $k$, the induced map $\Ho^k(X^{\an}, \Q) \to \Ho^k(U^{\an}, \Q)$ preserves the weight filtration.

	  Next, we argue that if the lemma holds for a single compactification $X$ of $U$, then it holds for all compactifications. By weak factorization (Example~\ref{ex:weak_factorization}), it is enough to prove the following claim: If $X$ is a compactification of $U$, and $X' \to X$ is either a root stack along a smooth divisor in the complement of $U$, or a blowup along a smooth closed substack in the complement of $U$, then the conclusion of the lemma holds for $X$ if and only if it holds for $X'$. The claim follows from considering the semiorthogonal decomposition associated to either a blowup or a root stack. 
	\end{proof}

	\begin{definition}
	\label{def:sb_lowest_weight}
	    Let $U$ be a smooth, separated scheme over $\C$, and let $\P \to U$ be a Severi--Brauer variety. The \emph{lowest-weight part} of $\ktop_i(\dperf(\P)_1)$ is the subgroup
	    \[
	    	\wtop_i(\dperf(\P)_1) \subset \ktop_i(\dperf(\P)_1)
	    \]
	    which is the intersection of $\ktop_i(\dperf(\P)_1))$ with $\wtop_i(\P)$. 
	\end{definition}

	\begin{theorem}
	\label{thm:open_twisted_mukai}
	    Let $U$ be a smooth, quasi-projective variety. Let $\scu \to U$ be an essentially topologically trivial $\bmu_n$-gerbe, and let $\P \to U$ be a Severi--Brauer variety which represents the Brauer class of $\scu$. Let $i$ be an integer.
	    \begin{enumerate} [label = (\arabic*)]
	    	\item $\wtop_i(\dperf(\P)_1)$ is a Hodge substructure of $\wtop_i(\P)$.
	    	\item For each $1$-twisted topological line bundle $L$ on $\scu$, there is an isomorphism of Hodge structures
	    	\[
	    		\wtop_i(\dperf(\P)_1) \to \wtop_i(U)^{c_1(L)}.
	    	\]
	    \end{enumerate}
	\end{theorem}

	\begin{proof}
		Following the proof of Theorem~\ref{thm:ktop_of_gerbe}, we consider the composition
		\[
			\begin{tikzcd}[column sep=large]
				\ktop_i(\dperf(\P)_0) \ar[r, "{[L(1)]}"] & \ktop_i(\dperf(\P)_1) \ar[r, "{[\sco_{\P_{\scu}}(-1)]}"] & \ktop_i(\dperf(\P)_0) \otimes \Q.
			\end{tikzcd}
		\]
		Multiplication by $[\sco_{\P_{\scu}}(-1)]$ gives an automorphism of the rational Hodge structure $\wtop_i(\P) \otimes \Q$, so (1) follows since the right-hand map is an isomorphism after extending scalars to $\Q$. Then (2) follows, since the left-hand map identifies the Hodge structure on $\wtop_i(\dperf(\P)_1))$ with the twisted Mukai structure $\wtop_i(U)^{c_1(L)}$.
	\end{proof}

	\subsection{Computations on quasi-projective varieties}
	\label{ssec:computations_on_quasi_proj}

	In this section, we prove several technical results which will be used later. In Lemma~\ref{lem:hodge_class_open_subvar}, we give a lower bound for the ranks of Hodge classes for twisted Mukai structures on smooth, quasi-projective varieties, and in Lemma~\ref{lem:affine_surface}, we compute the set of Hodge classes for twisted Mukai structures on affine surfaces.

	For a smooth, separated scheme $U$ over $\C$, we define
	\[
		\Wi_0\Ho^2(U^{\an}, \Z(1)) \subset \Ho^2(U^{\an}, \Z(1))
	\]
	to be the set of class $v$ whose image in $\Ho^2(U^{\an}, \Q(1))$ lies in $\Wi_0 \Ho^2(X^{\an}, \Q(1))$. The proof of Lemma~\ref{lem:basic_h2_v2} below implies that $\Wi_0 \Ho^2(U^{\an}, \Z(1))$ coincides with the image of $\Ho^2(X^{\an}, \Z(1))$ for any smooth compactification of $U$, although we do not need this fact.

	\begin{lemma}
	\label{lem:basic_h2_v2}
	    Let $U$ be a smooth, separated scheme over $\C$. Then
	    \[
	    	\NS(U) = \Hdg(\Wi_0 \Ho^2(U^{\an}, \Z(1))).
	    \]
	\end{lemma}

	\begin{proof}
	    Let $X$ be a smooth compactification of $U$ so that $D = X - U$ has simple normal crossings. Consider the diagram
	    \[
	    	\begin{tikzcd}
	    		0 \ar[r] & \NS(X) \ar[d, "a"] \ar[r, "f"] & \Ho^2(X^{\an}, \Z(1)) \ar[d, "b"] \ar[r] & \Coker f \ar[d, "c"] \ar[r] & 0 \\
	    		0 \ar[r] & \Hdg(\Wi_0\Ho^2(U^{\an}, \Z(1))) \ar[r, "g"] & \Ho^2(U^{\an}, \Z(1)) \ar[r] & \Coker g \ar[r] & 0 .
	    	\end{tikzcd}
	    \]
	    Since $a \otimes \Q$ is an isomorphism, we need to show that $\Coker a$ is torsion-free. It is well-known that $\Ker b$ is generated by the cycle classes of the components of $D$ (see for instance \cite[Prop. 2.2]{two_coniveau}). In particular, $\Ker a$ is isomorphic to $\Ker b$, and the connecting homomorphism $\Ker c \to \Coker a$ is injective. Since $f$ is the inclusion of a saturated sublattice, $\Coker f$ is torsion-free, hence so is $\Ker c$. In particular, to show that $\Coker a$ is torsion-free, it is enough to show that $\Coker b$ is torsion-free. 

	    Finally, the torsion-freeness of $\Coker b$ follows from the Leray spectral sequence for $j:U \to S$, and the observations that $\Ho^0((D^{(2)})^{\an}, \Z)$ and $\Ho^1((D^{(1)})^{\an}, \Z)$ are torsion-free, where $D^{(1)}$ is the disjoint union of the components of $D$ and $D^{(2)}$ is the disjoint union of the pairwise intersections of the components of $D$. 
	\end{proof}

	Let $U$ be a smooth, quasi-projective variety over $\C$. Given a class $v \in \Ho^2(U^{\an}, \Z(1))$ and an integer $n > 0$, let $\alpha(v, n) \in \Br(U)$ be the image of $v$ under the composition
	\[
		\begin{tikzcd}[column sep=large]
			\Ho^2(U^{\an}, \Z(1)) \ar[r, "\exp(-/n)"] & \Ho^2_{\et}(U, \bmu_n) \ar[r] & \Ho^2_{\et}(U, \G_m).
		\end{tikzcd}
	\]
	The following lemma shows that $\wtop_i(U)^{v/n}$ depends only on $\alpha(v, n)$.

	\begin{lemma}
	\label{lem:depends_only_on_br}
	    In the situation above, let $v_1, v_2 \in \Ho^2(U^{\an}, \Z(1))$, and let $n_1, n_2 > 0$ be integers. If $\alpha(v_1, n_1) = \alpha(v_2, n_2)$, then there is an isomorphism
	    \[
	    	\wtop_i(U)^{v_1/n_1} \simeq \wtop_i(U)^{v_2/n_2}.
	    \]
	\end{lemma}

	\begin{proof}
	    For simplicity, we write $\Wi(v/n) = \wtop_i(\dperf(U))^{v/n}$. We make three observations:
	    \begin{enumerate} [label = (\arabic*)]
	    	\item For $t > 0$, $\Wi(t  v/t  n) \simeq \Wi(v/n)$.
	    	\item For $w \in \Ho^2(U^{\an}, \Z(1))$, $\Wi((v + nw)/n) \simeq \Wi(v/n)$.
	    	\item For $w \in \NS(U)$, $\Wi((v + w)/n) \simeq \Wi(v/n)$.
	    \end{enumerate}
	    Note that (1) is from the definition, and for (2) and (3) one follows the proof of Corollary~\ref{cor:twisted_mukai-b_field}.

	    For the lemma, by (1) we may assume that $n_1 = n_2$. Then from the Kummer sequence, $v_1 - v_2$ lies in $n_1 \cdot \Ho^2(U^{\an}, \Z(1)) + \NS(U)$, and we apply (2) and (3).
	\end{proof}

	\begin{lemma}
	\label{lem:hodge_class_open_subvar}
	    Let $U$ be a smooth, quasi-projective variety over $\C$, with a class $v \in \Ho^2(U^{\an}, \Z(1))$ and an integer $n > 0$. Then the rank of any Hodge class in $\Wi_0(U)^{v/n}$ is divisible by $n_0$, where $n_0$ is the order of $\alpha(v, n)$ in $\Br(U)$.
	\end{lemma}

	\begin{proof}
	    By Lemma~\ref{lem:depends_only_on_br}, we may suppose that $n = n_0$. Let $w \in \wtop_0(U)^{v/n}$ be a Hodge class, and write $w = (w_0, w_2, \dots)$, so that
	    \[
	    	\exp(B) \cdot \ch(w) = (w_0, w_2 + w_0 \cdot B, \dots).
	    \]
	    Then $w_2 + w_0 \cdot (v/n)$ is Hodge, so $n  w_2 + w_0 v$ is an integral Hodge class in $\Wi_0(\Ho^2(U^{\an}, \Z(1))$, and hence is the Chern class of an algebraic line bundle by Lemma~\ref{lem:basic_h2_v2}. In particular, the image of  $w_0 v + n w_2$ under the composition
		\[
			\Ho^2(U^{\an}, \Z(1)) \to \Ho^2(U^{\an}, \bmu_n) \to \Br(U)[n]
		\]
		is trivial. On the other hand, the image coincides with $w_0 \cdot \alpha(v/n)$. Therefore, $n$ divides $w_0$.	    
	\end{proof}

	\begin{lemma}
	\label{lem:affine_surface}
	   Let $U$ be a smooth, affine surface over $\C$, with a class $v \in \Ho^2(U^{\an}, \Z(1))$ and an integer $n > 0$. There is a short exact sequence
	    \[
	    	\begin{tikzcd}
	    		0 \ar[r] & \NS(U) \ar[r] & \Hdg(\wtop_0(U)^{v/n}) \ar[r, "\rk"] & n_0 \cdot \Z \ar[r] & 0,
	    	\end{tikzcd}
	    \]
	    where $n_0$ is the order of $\alpha(v/n)$ in $\Br(U)$.
	\end{lemma}

	\begin{proof}
		From the Atiyah--Hirzebruch spectral sequence, the map
		\[
			(\rk, c_1):\ktop(U) \to \Ho^0(U^{\an}, \Z) \oplus \Ho^2(U^{\an}, \Z(1))
		\]
		is an isomorphism. By Lemma~\ref{lem:depends_only_on_br}, we may suppose that $n = n_0$. By Lemma~\ref{lem:basic_h2_v2}, $(w_0, w_2)$ is a Hodge class in $\wtop_0(U)^{v/n}$ if and only if $w_0 \cdot (v/n) + w_2$ lies in $\NS(U)_{\Q}$. If $w_0 = 0$, then $w_2 \in \NS(U)$ by Lemma~\ref{lem:basic_h2_v2}. The morphism 
		\[
			\NS(U) \to \Hdg(\wtop_0(U)^{v/n})
		\]
		in the statement of the lemma is given by sending $w_2$ to $(0, w_2)$. Since $(n, -v)$ is a Hodge class of rank $n$, it remains to show that the rank of any Hodge class is divisible by $n$, which is Lemma~\ref{lem:hodge_class_open_subvar}.
	\end{proof}

	\section{The Hodge-theoretic index} 
	\label{sec:hodge_theoretic_index}

	In \S \ref{ssec:period_index_problem}, we review the period-index conjecture for function fields over an algebraically closed field. In \S \ref{ssec:global_period_index_problem}, we review strategies for converting the period-index conjecture into a global conjecture about smooth proper varieties or orbifolds. In \S \ref{ssec:hodge_theoretic_index}, we introduce the Hodge-theoretic index, and establish some of its properties. In \S \ref{ssec:bounding_the_hodge_theoretic_index}, we obtain a bound on the Hodge-theoretic index for topologically trivial classes. In \S \ref{ssec:potential_obstructions}, we examine a potential obstruction to the period-index conjecture on threefolds.

	\subsection{The period-index problem}
	\label{ssec:period_index_problem}

	We refer to \cite{ct_bourbaki} for an account of the classical theory of Brauer groups and the period-index conjecture. 

	\begin{definition}
	\label{def:k_field_index}
	    Let $K$ be a field, and let $\alpha \in \Br(K)$. The \emph{period} of $\alpha$ is its order in the torsion group $\Br(K)$. The \emph{index} of $\alpha$ is equal to the following positive integers:
	    	\begin{enumerate} [label = \normalfont{(\arabic*)}]
			\item The minimum rank of an $\alpha$-twisted vector space. 
			\item The minimum degree of a separable extension $K'/K$ such that $\alpha_{K'} = 0$.
			\item The integer $d' + 1$, where $d'$ is the minimum dimension of a Severi--Brauer variety of class $\alpha$.
			\item The degree $\sqrt{\dim D}$ of the unique division algebra $D$ of class $\alpha$.
	\end{enumerate}
	\end{definition}

	One may show that $\ind(\alpha)$ divides $\per(\alpha)$, and that $\per(\alpha)$ and $\ind(\alpha)$ share the same prime factors. In particular, for each $\alpha \in \Br(K)$, there is an integer $\epsilon(\alpha)$ such that
	\[
	 	\ind(\alpha) \divides \per(\alpha)^{\epsilon(\alpha)}.
	 \] 
	 The \emph{period-index problem} for $K$ is the problem of determining a bound for $\epsilon(\alpha)$, preferably one which is uniform over the Brauer group. The main conjecture is the following:

	\begin{conjecture}[The period-index conjecture] \label{conj:period_index}
		Let $K$ be a field of transcendence degree $d$ over an algebraically closed field $k$. For any Brauer class $\alpha \in \Br(K)$, 
		\[
		 	\ind(\alpha) \divides \per(\alpha)^{d - 1}.
		 \] 
	\end{conjecture}

	\noindent
	There are variants of Conjecture~\ref{conj:period_index}, for instance when $k$ is a finite field or a $p$-adic field.

	The status of Conjecture \ref{conj:period_index} is as follows:
	\begin{enumerate} 
		\item[$d \leq 1$] The Brauer group $\Br(K)$ is trivial.
		\item[$d = 2$] Conjecture~\ref{conj:period_index} was proved by de Jong \cite{dejong_period_index} when $\per(\alpha)$ is prime to the characteristic of $k$. When the characteristic of $k$ divides $\per(\alpha)$, the conjecture was proved through different methods by de Jong--Starr \cite{dejong_starr} and Lieblich \cite{lieblich_twisted_sheaves}.
		\item[$d \geq 3$] Conjecture~\ref{conj:period_index} is not known for any given field $K$.
	\end{enumerate}

	When $d \geq 3$, it is not known in general that there exists an exponent $\epsilon$ such that for each $\alpha \in \Br(K)$,
	\[
		 \ind(\alpha) \divides \per(\alpha)^{\epsilon}.
	\]
	However, Matzri \cite{matzri} has obtained non-uniform exponents $\epsilon(\alpha)$, which depend only on $\trdeg(K/k)$ and the prime factors of $\per(\alpha)$, based on estimates for the symbol length of a central simple algebra of given period. Aside from the cases described above, Matrzi's bounds are the best available.

	\subsection{The global period-index problem}
	\label{ssec:global_period_index_problem}

	We now explain a version of the period-index conjecture for orbifolds over fields.

	\begin{definition}
	\label{def:period_on_dm_stack}
	    Let $X$ be a connected Deligne--Mumford stack of finite type over a field $k$, and let $\alpha \in \Ho^2_{\et}(X, \G_m)^{\tors}$. The \emph{period} $\per(\alpha)$ of $\alpha$ is its order in $\Ho^2_{\et}(X, \G_m)$. The \emph{index} $\ind(\alpha)$ of $\alpha$ is the positive generator of the image of a rank homomorphism,
			\[
				\K_0(\dperf(X, \alpha)) \to \Z
			\]
		which arises by viewing $\dperf(X, \alpha)$ as $\dperf^1(\scx)$ for $\scx \to X$ a $\bmu_n$-gerbe of class $\alpha$. 
	\end{definition}

	We recall our convention that an \emph{orbifold} over a field $k$ is a Deligne--Mumford stack of finite type over $k$ with trivial generic stabilizers. 

	\begin{lemma}
	\label{lem:ind_on_base_field}
	    Let $X$ be a smooth, connected orbifold over $k$. If $\alpha \in \Ho^2_{\et}(X, \G_m)$, then 
	    \[
	    	\ind(\alpha) = \ind(\alpha_K),
	    \]
	    where $\alpha_K$ is the image of $\alpha$ in $\Br(K)$.
	\end{lemma}

	\begin{proof}
	    We first show that $\ind(\alpha_K)$ divides $\ind(\alpha)$. Let $\scx \to X$ be a $\G_m$-gerbe of class $\alpha$, and let $\scx_K \to \Spec K$ be the fiber over the generic point. Choose a $1$-twisted vector bundle $V_K$ on $\scx_K$ of rank $\ind(\alpha_K)$. By \cite[Prop. 3.1.1.9]{lieblich_twisted_sheaves}, there exists a coherent extension of $V_K$ to a $1$-twisted coherent sheaf $V$ on $\scx$, and since $\scx$ is smooth, $V$ is quasi-isomorphic to a perfect complex $E$, with $\rk E = \ind(\alpha_K)$. We next show that $\ind(\alpha)$ divides $\ind(\alpha_K)$. Let $E$ be a $1$-twisted perfect complex on $\scx$ of rank $r$. Then the restriction of $E$ to $\scx_K$ is perfect, and the rank of any $1$-twisted locally free sheaf on $\scx_K$ is divisible by $\ind(\alpha)$.
	\end{proof}

	\begin{conjecture}[Global period-index] \label{conj:global_period_index}
		Let $X$ be a smooth, proper, connected orbifold of dimension $d$ over an algebraically closed field $k$. For any $\alpha \in \Ho^2_{\et}(X, \G_m)$,
		\[
		 	\ind(\alpha) \divides \per(\alpha)^{d - 1}.
		 \] 
	\end{conjecture}

	\begin{remark}
		The global period-index conjecture over $\C$ (Conjecture~\ref{conj:global_period_index}) is equivalent to the period-index conjecture for complex function fields (Conjecture~\ref{conj:period_index}). Lemma~\ref{lem:ind_on_base_field} shows that the conjecture for function fields implies the global conjecture. In the other direction, let $K$ be a function field over $\C$. Then any Brauer class extends to a quasi-projective model of $K$, and one may apply Lemma~\ref{lem:root_stack_brauer_class} to pass to a smooth, proper orbifold.

		Alternatively, the discriminant-avoidance theorem of de Jong and Starr \cite{dejong_starr} implies that, in order to prove the period-index conjecture for complex function fields of transcendence degree $\leq d$, it is enough to verify the global period index conjecture for smooth, complex, projective varieties of dimension $\leq d$. However, the reduction is inexplicit. 
	\end{remark}

	\subsection{The Hodge-theoretic index}
	\label{ssec:hodge_theoretic_index}

	Let $X$ be a smooth, proper, connected orbifold over $\C$, and let $\alpha \in \Ho_{\et}^2(X, \G_m)$. Given a closed point $\Spec \C \to X$, the pullback
	\[
		\dperf(X, \alpha) \to \dperf(\Spec \C)
	\]
	defines a rank homomorphism $\ktop_0(X, \alpha) \to \ktop_0(\Spec \C) = \Z$, compatible with the rank homomorphism $\K_0(\dperf(X, \alpha)) \to \Z$ on algebraic $K$-theory. 

	\begin{definition}
	\label{def:hodge_theoretic_index}
	    The \emph{Hodge-theoretic index} $\hind(\alpha)$ is the positive generator of the image of the rank homomorphism
	\[
		\Hdg(\dperf(X, \alpha), \Z) \to \Z.
	\]
	\end{definition}

	Observe that $\hind(\alpha)$ divides $\ind(\alpha)$, and if they differ, then the integral Hodge conjecture fails for $\dperf(X, \alpha)$. In general, the Hodge-theoretic index may differ from both the period and the index, but it enjoys a number of the same formal properties, as indicated in Lemma~\ref{lem:properties_of_hdg_index} below.

	We first establish some basic properties of $\hind(\alpha)$. Since $\Br(X)$ is a torsion group, any nonzero element $\alpha$ admits a prime decomposition
	\[
		\alpha = \alpha_1 + \alpha_2 + \cdots + \alpha_n,
	\]
	characterized by the property that $\per(\alpha_i) = \ell_i^{r_i}$, for distinct primes $\ell_1, \dots, \ell_n$ and $r_i > 0$. 

	\begin{lemma}
	\label{lem:properties_of_hdg_index}
	    Let $X$ be a smooth, proper, connected orbifold over $\C$ with function field $K$, and let $\alpha \in \Ho^2_{\et}(X, \G_m)$.
	    \begin{enumerate} [label = (\arabic*)]
	    	\item If $f:X' \to X$ is generically finite, then $\hind(\alpha) \divides (\deg f) \cdot \hind(f^* \alpha)$.
	    	\item If $X'$ is smooth, proper connected orbifold over $\C$ which is birational to $X$, and $\alpha_K \in \Br(K)$ extends to a class $\alpha' \in \Ho^2_{\et}(X', \G_m)$, then $\hind(\alpha) = \hind(\alpha')$.
			\item If $\alpha = \alpha_1 + \cdots + \alpha_n$ is the prime decomposition of $\alpha$, then $\hind(\alpha) = \prod_i \hind(\alpha_i)$. 
			\item $\per(\alpha) \divides \hind(\alpha) \divides \ind(\alpha)$.
	    \end{enumerate}
	\end{lemma}

	\begin{proof}
		For (1), observe that
    \[
    	\rk f_* f^* v = (\deg f) \rk v
    \]
    for any $v \in \ktop_0(S, \alpha)$. Then (2) follows from (1) and weak factorization for a $\bmu_n$-gerbe of Brauer class $\alpha$, which (as noted in Example~\ref{ex:weak_factorization}) is compatible with the gerbe structure.

     For (3), it is not hard to see that $\hind(\alpha)$ divides $\prod \hind(\alpha_i)$, so we show that for each $i$, $\hind(\alpha_i)$ divides $\hind(\alpha)$. Let $E$ be an object in $\dperf(X, \alpha_i - \alpha)$ of rank $\ind(\alpha_i - \alpha)$. Since
     \[
     	\Fun_X(\dperf(X, \alpha), \dperf(X, \alpha_i)) \simeq \dperf(X, \alpha_i - \alpha),
     \]
     the object $E$ gives a functor from $\dperf(X, \alpha)$ to $\dperf(X, \alpha_i)$, and the induced map on topological $K$-theory sends a Hodge class of rank $\hind(\alpha)$ to a Hodge class of rank $\hind(\alpha) \cdot \ind(\alpha - \alpha_i)$. Since it is clear that $\hind(\alpha_i)$ divides $\ind(\alpha_i)$, $\hind(\alpha_i)$ is prime to $\ind(\alpha - \alpha_i)$, and so $\hind(\alpha_i)$ divides $\hind(\alpha)$.

    For (4), it is clear that $\hind(\alpha)$ divides $\ind(\alpha)$, so we show that $\per(\alpha)$ divides $\hind(\alpha)$. Let $U \subset X$ be a substack which is equivalent to a smooth quasi-projective variety. By Lemma~\ref{lem:bloch_kato}, we may suppose (after perhaps shrinking $U$) that the restriction of $\alpha$ to $U$ is topologically trivial. By Lemma~\ref{lem:open_subvariety_gerbe} and (2), we may assume that there is a Severi--Brauer variety $\P \to X$ whose restriction over $U$ has class $\alpha_U$.

   	It follows from Lemma~\ref{lem:lowest_weight_does_not_depend} that the induced map
    \[
    	\ktop_0(\dperf(\P)_1) \to \wtop_0(\dperf(\P_U)_1).
    \]
    is a morphism of Hodge structures. We conclude by observing that the rank of any Hodge class in the right-hand side is divisible by $\per(\alpha)$, by Theorem~\ref{thm:open_twisted_mukai} and Lemma~\ref{lem:hodge_class_open_subvar}.
	\end{proof}

	\begin{example} \label{ex:kresch}
		Let $X$ be a smooth, proper variety, and let $v \in \Ho^2(X^{\an}, \Z(1))$ be a cohomology class. In \cite{kresch_quarternion}, Kresch shows that if the image $\exp(v/2)$ of $v$ in $\Br(X)[2]$ has index $2$, then there exists a Hodge class $h \in \Hdg^4(X^{\an}, \Z(2))$ such that 
		\begin{equation} \label{eq:kresch_obstr}
			v^2 \equiv h \mod 2.
		\end{equation}
		Kresch observes that \eqref{eq:kresch_obstr} imposes a nontrivial condition on $v$, and so constructs examples of threefolds with Brauer classes of period $2$ and index $4$.

		From our calculation of the twisted Mukai structure, one may give a quick proof of Kresch's result. Let $B = v / 2$. If the index of $\alpha$ is $2$, then let 
		\[
			\ch(w) = (2, w_2, w_4, \dots)
		\]
		be the Chern character of a Hodge class $w$ of rank $2$ in $\ktop_0(X)^B$. Since $w$ is Hodge, 
		\[
			\exp(B) \cdot \ch(w) = (2, 2B + w_2 , B^2 + w_2 \cdot B + w_4, \dots)
		\]
		is Hodge in each degree. Since $2w_4$ is integral, 
		\begin{align*}
			4(B^2 + w_2 \cdot B + w_4) &= v^2 + 2 v \cdot w_2 + 4 \cdot w_4 \\
				&\equiv v^2 \mod 2,
		\end{align*}
		which recovers Kresch's obstruction \eqref{eq:kresch_obstr}.
	\end{example}

	\begin{remark}
		We briefly compare the Hodge-theoretic index with the \'etale index $\etind(\alpha)$ studied by Antieau \cite{antieau_divisibility},  \cite{antieau_et_ind}, and Antieau--Williams \cite{antieau_williams_godeaux}. Let $X$ be a smooth, proper scheme over $\C$, and let $\K_0^{\et}(X)$ be the \'etale $K$-theory of $X$. There is a sequence of morphisms
		\begin{equation} \label{eq:etale_index_rank}
			\begin{tikzcd}
				\K_0(X) \ar[r, "s"] & \K^{\et}_0(X) \ar[r] & \ktop_0(X) \ar[r, "\rk"] & \Z.
			\end{tikzcd}
		\end{equation}
		From \cite[Theorem 2.15]{thomason}, $s$ becomes an isomorphism after tensoring with $\Q$. It follows that the image of $\K_0^{\et}(X)$ in $\ktop_0(X)$ is contained in $\Hdg(\dperf(X), \Z)$. 

		Let $\alpha \in \Br(X)$ be a Brauer class, and let $\P \to X$ be a Severi--Brauer variety of class $\alpha$. As a summand of the sequence \eqref{eq:etale_index_rank} on $\P$, we obtain a sequence
		\begin{equation} \label{eq:etale_index_twisted}
			\begin{tikzcd}
				\K_0(X, \alpha) \ar[r] & \K^{\et}_0(X, \alpha) \ar[r] & \ktop_0(\dperf(X, \alpha)) \ar[r, "\rk"] & \Z.
			\end{tikzcd}
		\end{equation}
		From above, $\K^{\et}(X, \alpha)$ maps into $\Hdg(\dperf(X, \alpha), \Z)$.

		In \cite{antieau_et_ind}, Antieau defines the \emph{\'etale index} $\ind_{\et}(\alpha)$ of $\alpha$ to be the positive generator of $\rk(\K_0^{\et}(S, \alpha))$. In \cite{antieau_williams_godeaux}, Antieau and Williams construct Brauer classes on Serre-Godeaux varieties with $\per(\alpha)$ strictly less than $\ind_{\et}(\alpha)$.

		The Hodge-theoretic index provides a straightforward method for finding such examples. From the discussion above,
		\[
			\hind(\alpha) \divides \ind_{\et}(\alpha),
		\]
		so it suffices to construct examples where $\hind(\alpha)$ exceeds $\per(\alpha)$. For instance, any pair $(X, \alpha)$ such that $\alpha \in \Br(X)[2]$ is obstructed by Kresch's method (Example \ref{ex:kresch}) furnishes an example. Alternatively, it is straightforward to construct Brauer classes on abelian varieties of dimension $d \geq 3$ with $\hind(\alpha) > \per(\alpha)$, using Corollary~\ref{cor:twisted_mukai-b_field}.
	\end{remark}

	\subsection{Bounding the Hodge-theoretic index}
	\label{ssec:bounding_the_hodge_theoretic_index}

	In this section, we obtain an upper bound for the Hodge-theoretic index of a topologically trivial Brauer class.

	\begin{theorem}
	\label{thm:bounding_the_hodge_index}
	    Let $X$ be a smooth, proper variety of dimension $d$ over $\C$. For any topologically trivial Brauer class $\alpha \in \Br(X)$ with $\per(\alpha) = n$,
	    \[
	    	\hind(\alpha) \divides n^{d - 1} \cdot ((d - 1)!)^{d - 2}.
	    \]
	    In particular, if $n$ is prime to $(d - 1)!$, then $\hind(\alpha) \divides n^{d - 1}$.
	\end{theorem}

	\begin{remark}
		Theorem~\ref{thm:bounding_the_hodge_index} is a Hodge-theoretic analogue of a recent result of Antieau--Williams on the topological period-index conjecture \cite{top_per_ind}.
	\end{remark}

	We begin with a simple result on the image of topological $K$-theory in integral cohomology.

	\begin{lemma}
	\label{lem:exp_lemma}
	    Let $M$ be a finite CW complex of dimension $2d$. For any $v \in \Ho^2(M, \Z)$ and any polynomial $f(t) \in  \Z[t]$, there exists a constant $c \in \Q$ so that 
	    \[
	    	((d - 1)!)^{d - v_0(f) - 1}(f(v) + c \cdot v^d) \in \Ho^{\ev}(M, \Q)
	    \]
	    lies in the image of $\ktop_0(M)$ under the Chern character, where $v_0(f)$ is the order of vanishing of $f(t)$ at $0$.

	\end{lemma}

	\begin{proof}
	    Since $\CP^{\infty} = \K(\Z, 2)$, it suffices to treat the case of its $2d$-skeleton $M = \CP^{d}$, $v = c_1(\sco_{\CP^d}(1))$. We use the following fact: If $x \in \Ho^{2k}(M, \Z)$ is an integral class, then there exists $z \in \ktop_0(M)$ such that
	    \[
	    	\ch(z) = x + \ch_{k + 1}(z) + \ch_{k + 2}(z) + \cdots
	    \]
	    The fact is a consequence of the degeneration of the Atiyah--Hirzebruch spectral sequence for $M$ at $\E_2$ \cite{atiyah_hirz}.

	    We proceed by ascending induction on $k = v_0(f)$. If $k = d - 1$, then the result is clear, since (from the fact) we may find $z \in \ktop_0(M)$ so that $f(v)$ and $\ch(z)$ differ by $c \cdot v^d$. If $k < d - 1$, then we may apply the fact to obtain a class $z \in \ktop_0(M)$ so that the polynomial $f(v) - \ch(z) \in \Q[v]$ vanishes to order $k + 1$ at $0$. Since $(d - 1)!\cdot \ch(z)$ is integral except in the top degree, there exists a constant $c_0 \in \Q$ so that
	    \[
	    	g(v) = (d - 1)! \cdot (f(v) - \ch(z)) + c_0 v^d
	    \]
	    is integral, and $v_0(g) = k + 1$. From the inductive hypothesis, there exists $c_1 \in \Q$ so that
	    \begin{align*}
	    	((d - 1)!)^{d - k - 2} \cdot (g(v) + c_1 v^d) &=  ((d - 1)!)^{d - k - 1}\left(f(v) - \ch(z) + \frac{1}{(d - 1)!} (c_0 + c_1) v^d \right) 
	    \end{align*}
	    lies in the image of $\ktop_0(M)$.
	\end{proof}

	\begin{proof}[Proof of Theorem~\ref{thm:bounding_the_hodge_index}]
	    Let $v \in \Ho^2(X^{\an}, \Z(1))$ be an integral class so that $v/n$ is a rational $B$-field for $\alpha$. Let
	    \begin{align*}
	    	f(v) &= n^{d - 1}\exp(-v/n) - \left(\frac{(-1)^d}{n \cdot d!} + \frac{(n^{d - 2})^d}{d!}\right)v^d \\
	    		&= n^{d - 1} - n^{d - 2}v + \frac{n^{d - 3}}{2} v^2 - \cdots + \frac{(-1)^{d - 1}}{(d - 1)!} v^{d - 1} - \frac{(n^{d - 2})^d}{d!} v^d
	    \end{align*}
	    For each $c \in \Q$, the class $((d - 1)!)^{d - 2} \cdot f(v) + cv^d$ is a rational Hodge class in $\ktop_0(X)^B \otimes \Q$. We claim that it lies in the image of the Chern character for an appropriate choice of $c \in \Q$. 

	    The polynomial 
	    \begin{align*}
	    	g(v) &= (d - 1)! \cdot \left(f(v) - n^{d - 1} + (\exp(n^{d - 2}v) - 1)\right) \\ 
	    		&= (d - 1)! \cdot \sum_{k = 2}^{d - 1} \frac{(n^{d - 2})^k + (-1)^k n^{d - k - 1}}{k!} \cdot v^k
	    \end{align*}
	    is integral, and $v_0(g) = 2$. By Lemma~\ref{lem:exp_lemma}, there exists a constant $c' \in \Q$ such that 
	    \[
	    	((d - 1)!)^{d - 3}\left( g(v) + c'v^d\right)
	   	\] 
	   	lies in the image of the Chern character, which implies the claim.
	\end{proof}

	\begin{remark}
	\label{rem:improve_bound}
		The factor of $((d - 1)!)^{d - 2}$ appearing in Theorem~\ref{thm:bounding_the_hodge_index} is not optimal. One could improve it by incorporating a more careful analysis of the image of $\ktop_0(\CP^d)$ in $\Ho^{\ev}(\CP^{d, \an}, \Q)$ into the proof of Lemma~\ref{lem:exp_lemma}, or, alternatively, with the integrality theorem of Adams \cite{adams}. Both routes lead to better but more complicated bounds, which unfortunately still contain a factor of $(d - 1)!$. 
	\end{remark}

	\subsection{Potential obstructions to period-index bounds}
	\label{ssec:potential_obstructions}

		Let $X$ be a smooth, proper threefold, and let $\alpha \in \Br(X)$ be a topologically trivial Brauer class. Then Theorem~\ref{thm:bounding_the_hodge_index} provides the bound
		\[
			\hind(\alpha) \divides 8,
		\]
		whereas the period-index conjecture would imply that $\hind(\alpha)$ divides $4$. The goal of this section is to analyze the situation in detail.

		\begin{theorem}
		\label{thm:obstruction_threefold}
		    Let $X$ be a smooth, proper variety over $\C$ with $v \in \Ho^2(X^{\an}, \Z(1))$, and let $\alpha$ be the image of $v/2$ in $\Ho^2_{\et}(X, \G_m)[2]$. If $\hind(\alpha)$ divides $4$, then there exists $H \in \NS(X)$ such that
		    \begin{equation}
		    \label{eq:obstr}
		     	v^2 + v \cdot H \in \Ho^4(X^{\an}, \Z(2))
		     \end{equation}
		     is congruent to a Hodge class modulo $2$.
		\end{theorem}

		\begin{proof}
		    If $\pi:X' \to X$ is a blowup along a smooth subvariety, then $v$ satisfies the conclusion of the theorem if and only if $\pi^* v$ does. Therefore, we may suppose that $X$ is projective, so that $\Br(X) = \Ho^2_{\et}(X, \G_m)$ and Corollary~\ref{cor:twisted_mukai-b_field} applies.

		    Suppose that $w \in \ktop_0(X)^B$ is a Hodge class of rank $4$. Then
			\[
				\exp(B) \cdot (4, w_2, w_4, \dots) = (4, 4B + w_2, 2B^2 + B \cdot w_2 + w_4, \dots)
			\]
			is Hodge in each degree, where $\ch(w) = (4, w_2, w_4, \dots)$. We may write:
			\begin{itemize}
				\item $H = 4B + w_2$, for $H \in \NS(X)$.
				\item $w_4 = \frac{1}{2}w_2^2 +  \epsilon$, where $\epsilon$ is an integral class.
			\end{itemize}
			The second point comes from the condition that $(4, w_2, w_4, \dots)$ lies in the image of $K$-theory, and the leading terms of Chern characters are integral.

		Let $Z = 2B^2 + B \cdot w_2 + w_4$. Then $Z$ is a rational Hodge class, and
		\begin{align*}
			 2Z &= v^2 + v(H - 2v) + (H - 2v)^2 + 2 \epsilon \\
			 	&\equiv v^2 + v \cdot H - H^2 \mod 2.
		\end{align*}
		In particular, $v^2 + v \cdot H$ is congruent to a Hodge class modulo $2$.
		\end{proof}

	\begin{question}
	\label{q:prime2}
		Let $X$ be a smooth, proper threefold over $\C$, and let $v \in \Ho^2(X^{\an}, \Z(1))$. Does there exist a class $H \in \NS(X)$ such that $v^2 + H \cdot v$ is congruent modulo $2$ to a Hodge class in $\Ho^4(X^{\an}, \Z(2))$?
	\end{question}

	A negative answer to Question~\ref{q:prime2} would imply that the period-index conjecture fails for Brauer classes of period $2$ on complex threefolds. 

	\begin{remark}
		The potential obstruction from Theorem~\ref{thm:obstruction_threefold} is reminiscent of a topological obstruction for the period-conjecture for period $2$ Brauer classes on threefolds proposed by Antieau and Williams \cite{antieau_williams_6_complex}, which was later shown to be ineffective \cite{spinc}.
	\end{remark}

\section{Counterexamples to the integral Hodge conjecture} 
	\label{sec:counterexamples_to_the_integral_hodge_conjecture}

	The goal of this section is to illustrate a method for producing non-algebraic integral Hodge classes in the following contexts: 
	\begin{enumerate} [label = \normalfont{(\roman*)}]
		\item The classical integral Hodge conjecture on a Severi--Brauer variety of class $\alpha$.
		\item The categorical integral Hodge conjecture for $\dperf(X, \alpha)$. 
	\end{enumerate}
	In both cases, the method is to begin with a Brauer class $\alpha$ such that $\per(\alpha) < \ind(\alpha)$, and then construct an integral Hodge class whose algebraicity would imply that $\per(\alpha) = \ind(\alpha)$. In \S \ref{ssec:abelian_threefolds}, we obtain counterexamples for Severi--Brauer varieties on abelian threefolds, and in \S \ref{ssec:products_of_curves}, we obtain counterexamples on $\dperf(X, \alpha)$, for $X$ a product of curves.

	\subsection{Abelian threefolds}
	\label{ssec:abelian_threefolds}

	By Gabber's method \cite[Appendice]{CT_gabber}, one may construct Brauer classes on very general abelian threefolds of period $2$ and index $4$. Alternatively, one may use twisted Mukai structures to produce examples of Brauer classes on abelian threefolds with 
	\[
		2 = \per(\alpha) < \hind(\alpha) = 4
	\]
	The following theorem shows that all such examples give rise to counterexamples to the integral Hodge conjecture:

	\begin{theorem}
	\label{thm:abelian_threefold}
	    Let $X$ be an abelian threefold over $\C$, and let $\alpha \in \Br(X)$ be a Brauer class with $\per(\alpha) = 2$ and $\ind(\alpha) > 2$. For any Severi--Brauer variety $\P \to X$ of class $\alpha$ and relative dimension $d$, the integral Hodge conjecture fails in $\Ho^{2d}(\P^{\an}, \Z(d))$.
	\end{theorem}

	The integral Hodge conjecture for abelian threefolds is a result of Grabowski \cite{grabowski}, so Theorem~\ref{thm:abelian_threefold} shows that the integral Hodge conjecture for a Severi--Brauer variety $\P \to X$ may fail even if it holds for the base. 

	\begin{proof}
	    Let $\pi:\P \to X$ be a Severi--Brauer variety of class $\alpha$ and relative dimension $d \geq 3$. The relative Picard group of $\P/X$ is generated by a line bundle $\sco_{\P}(2)$, whose restriction to the fiber over each $s \in X(\C)$ is $\sco_{\P^n_s}(2)$. Let $Q \in \Ho^2(\P^{\an}, \Z(1))$ be the first Chern class of $\sco_{\P}(2)$. 

	    Since $\alpha$ is topologically trivial, $\P^{\an} \to X^{\an}$ is a topological projective bundle, and there exists a class $H \in \Ho^2(\P^{\an}, \Z(1))$ whose restriction to each closed fiber is the hyperplane class. We may write
	    \[
	    	H = \frac{1}{2}Q + \frac{1}{2} \pi^* v,
	    \]
	    for a class $v \in \Ho^2(X^{\an}, \Z(1))$. Since $Q$ is algebraic, the class 
	    \[
	    	\exp(- \pi^* v/2) \cdot \exp(H) = (1, w_2, w_4, \dots)  \in \bigoplus_k \Ho^{2k}(X^{\an}, \Q(k))
	    \]
	    is Hodge in each degree. In particular, the component
	    \[
	    	w_{d} = \frac{1}{d!}H^{d} - \frac{1}{2(d - 1)!} (\pi^*v) H^{d - 1} + \frac{1}{2^2 \cdot 2(d - 2)!}  (\pi^*v^2) H^{d - 2} - \frac{1}{2^3 \cdot 3!(d - 3)!}  (\pi^*v^3) H^{d - 3}
	    \]
	    is a rational Hodge class. Consider the class
	    \begin{align*}
	    	Z &= 2 \cdot d! \cdot \left(w_d + \frac{1}{2^3 \cdot 3!(d - 3)!}(\pi^* v^3) H^{d - 3} \right) \\
	    		&= 2H^d - d (\pi^* v) H^{d - 1} + \frac{(d)(d - 1)}{4} (\pi^*v^2) H^{d - 2}.
	    \end{align*}
	    First, $Z$ is a rational Hodge class, since $w_d$ and $(\pi^* v^3)H^{d - 3}$ are Hodge classes, as $v^3$ is Hodge. Second, $Z$ is integral, since $v^2$ is divisible by $2$, as is any degree $2$ cohomology class on an abelian variety.

	    It remains to argue that $Z$ is not algebraic. The Gysin homomorphism
	    \[
	    	\Ho^{2d}(\P^{\an}, \Z(d)) \to \Ho^{6}(X^{\an}, \Z(6))
	    \]
	    sends $Z$ to $2 \cdot [X]$, where $[X]$ is the fundamental class. If $Z$ is algebraic, then there exists a zero-cycle of degree $2$ on the generic fiber $\P_{\eta}$ of $\pi$. However, $\ind(\alpha)$ coincides with the minimum degree of a zero-cycle on $\P_{\eta}$.
	\end{proof}

	\begin{remark}
		The assumption that $X$ is an abelian threefold, as opposed to an arbitrary threefold, enters in only two places: First, so that one may arrange that $\ind(\alpha) > 2$, and second, to ensure that $v^2$ is divisible by $2$.
	\end{remark}

	\begin{remark}
	\label{rem:sb_ihc}
		If $X$ is a smooth, proper variety over $\C$, the truth of the integral Hodge conjecture for a Severi--Brauer variety $\P \to X$ depends only on the subgroup of $\Br(X)$ generated by $[\P]$. 

		Indeed, in the trivial case when $\P \to X$ is a projective bundle, then the integral Hodge conjecture for $\P$ is equivalent to the integral Hodge conjecture for $X$. In the general case, if $\P$ and $\P'$ are Severi--Brauer varieties which generate the same subgroup of $\Br(X)$, then we may consider the diagram
		\[
			\begin{tikzcd}
				\P \times_X \P' \ar[d, "\pi_1"] \ar[r, "\pi_2"] & \P' \ar[d] \\
				\P \ar[r] & X.
			\end{tikzcd}
		\]
		We observe that $\pi_1$ and $\pi_2$ are projective bundles.
	\end{remark}

 	\subsection{Products of curves}
 	\label{ssec:products_of_curves}

	Let $C$ be a smooth, projective curve of genus $g \geq 2$ over $\C$, and let $E_1, \dots, E_k$ be elliptic curves over $\C$, with $2 \leq k \leq g$. Consider a symplectic basis
	\[
		s_1, \dots, s_g, t_1, \dots, t_g \in \Ho^1(C^{\an}, \Z),
	\]
	and for each $i$, let $u_i \in \Ho^1(E_i^{\an}, \Z)$ be a nonzero class. Let $X = C \times \prod_{i = 1}^{k} E_i$, and define the class
	\[
		v = 2 \pi i \cdot \sum_{i = 1}^k s_i \cup u_i \in \Ho^2(X^{\an}, \Z(1)).
	\]
	For each prime $\ell$, let $\alpha_{\ell} \in \Br(X)[\ell]$ be the Brauer class given by the rational $B$-field $v/\ell$.

 	\begin{theorem}
 	\label{thm:products_of_curves}
 	    For a prime $\ell$, let $X = C \times \prod_{i = 1}^k E_i$ and $\alpha_{\ell} \in \Br(X)$ be as above. If $C, E_1, \dots, E_k$ are very general, then the integral Hodge conjecture fails for $\dperf(X, \alpha)$.
 	\end{theorem}

 	\begin{proof}
 	    According to Gabber's result \cite[Appendice]{CT_gabber}, $\ind(\alpha_{\ell}) = \ell^{k}$. Therefore, it suffices to show that $\hind(\alpha) = \ell$. 

 	    We observe that $v^2 = 0$. If $L$ be a topological line bundle on $X$ with $c_1(L) = v$, then consider the class
 	    \[
 	     	w = \ell + (1 - [L]) \in \ktop_0(X), \hspace{5mm} \ch(w) = (\ell, -v, 0, \dots, 0).
 	     \] 
 	    Since $\exp(v/\ell)\ch(w)$ is Hodge, $w$ is a Hodge class for the twisted Mukai structure $\ktop_0(X)^{v/\ell}$.
 	\end{proof}

 	\begin{corollary}	
 	\label{cor:products_of_curves}
 	    Let $\ell$ be a prime. In the context of Theorem~\ref{thm:products_of_curves}, if $\P \to X$ is a Severi--Brauer variety of class $\alpha_{\ell}$, then the integral Hodge conjecture fails for $\P$.
 	\end{corollary}

 	\begin{proof}
 	    The cohomology ring of $X$ is torsion-free, so the integral Hodge conjecture for $\P$ in all degrees implies the integral Hodge conjecture for $\dperf(\P)$ (Remark~\ref{rem:implications_ihc}).
 	\end{proof}

\section{The integral Hodge conjecture for DM surfaces} 
	\label{sec:the_integral_hodge_conjecture_for_DM_surfaces}

	In \S \ref{ssec:preliminaries_on_dm_curves}, we obtain some preliminary results on the topological $K$-theory of Deligne--Mumford curves. In \S \ref{ssec:dm_surfaces}, we prove the integral Hodge conjecture for Deligne--Mumford surfaces.

	\subsection{Preliminaries on DM curves}
	\label{ssec:preliminaries_on_dm_curves}

	Let $C$ be a smooth, proper Deligne--Mumford stack of pure dimension $1$ over $\C$. Then $C$ is a gerbe over its total rigidification $C^{\rig}$ (Example~\ref{ex:total_rig}), which a smooth, proper orbifold. By Corollary~\ref{cor:non-abelian_gerbe_cover} and Tsen's theorem, there exists a representable finite \'etale cover $C' \to C^{\rig}$ and an equivalence 
	\begin{equation} \label{eq:curve_reduction}
		\dperf(C) \cong \dperf(C').
	\end{equation}
	As in the classical setting, $\ktop_0(\dperf(C))$ is spanned by algebraic classes:

	\begin{lemma}
	\label{lem:ihc_curve_surjective}
	    Let $C$ be a smooth, proper Deligne--Mumford curve over $\C$. Then
	    \[
	    	\K_0(\dperf(C)) \longrightarrow \ktop_0(\dperf(C))
	    \]
	    is surjective.
	\end{lemma}

	\begin{proof}
	    From \eqref{eq:curve_reduction}, it suffices to consider the case when $C$ is a smooth, proper orbifold. Let $p_1, \dots, p_r$ be the orbifold points of $C$. The coarse space $C^{\cs}$ is smooth, and from the semiorthogonal decomposition of a root stack (\S \ref{ssec:root_stacks}), 
	    \[
	    	\dperf(C) = \langle \dperf(C^{\cs}), \scc \rangle,
	    \]
	    where $\scc$ admits a semiorthogonal decomposition into copies of $\dperf(p_i)$. Since the desired result holds for $\dperf(\Spec \C)$ and $\dperf(C^{\cs})$, it holds for $\dperf(C)$.
	\end{proof}

	\begin{lemma}
	\label{lem:dm_curve}
	    Let $D$ be a simple normal crossing Deligne--Mumford curve over $\C$. Then 
	    \[
	    	\K_0(\dcoh(D)) \to \ktop_0(\dcoh(D))
	    \]
	    is surjective.
	\end{lemma}

	\begin{proof}
	    Let $D^\circ \subset D$ be the smooth locus of $D$, and let $Z = D - D^{\circ}$ be the complement. Note that $Z$ is smooth. From Lemma~\ref{lem:excision_sequence}, there is a diagram with exact rows
	    \[
	    	\begin{tikzcd}
	    		\K_0(\dperf(Z)) \ar[d, "a", "\sim"'] \ar[r] & \K_0(\dcoh(D)) \ar[d, "b"] \ar[r] & \K_0(\dperf(D^{\circ})) \ar[r] \ar[d, "c"] & 0 \\
	    		\ktop_0(\dperf(Z)) \ar[r] & \ktop_0(\dcoh(D)) \ar[r] & \ktop_0(\dperf(D^{\circ})) \ar[r] & 0.
	    	\end{tikzcd}
	    \]
	    The surjectivity on the right and the fact that $a$ is an isomorphism follow from the observation that $\dperf(Z)$ is equivalent to a product of copies of $\dperf(\Spec \C)$. 

	    Our goal is to show that $b$ is surjective. Since $a$ is an isomorphism, it suffices to show that $c$ is surjective. But $D^{\circ}$ is open inside of $D' = \coprod_i D_i$, where each $D_i$ is a gerbe over an irreducible component $D_i^{\cs}$ of $D^{\cs}$. We apply Lemma~\ref{lem:excision_sequence} again to obtain an exact sequence
	    \[
	    	\begin{tikzcd}
	    		\ktop_0(\dperf(Z')) \ar[r] & \ktop_0(\dperf(D')) \ar[r] & \ktop_0(\dperf(D^{\circ})) \ar[r] & 0,
	    	\end{tikzcd}
	    \]
	    where $Z'$ is a finite set of stacky points. By Lemma~\ref{lem:ihc_curve_surjective}, $\ktop_0(\dperf(D'))$ is spanned by algebraic classes, so the same holds for $\ktop_0(\dperf(D^{\circ}))$.
	\end{proof}

	\subsection{DM surfaces}
	\label{ssec:dm_surfaces}

	Our goal is to prove the following theorem:

	\begin{theorem}
	\label{thm:ihc_for_dm_surfaces}
	    Let $X$ be a smooth, proper Deligne--Mumford surface over $\C$. Then the integral Hodge conjecture holds for $\dperf(X)$.
	\end{theorem}

	The strategy of the proof is to reduce to the case of $\dperf(X, \alpha)$, for a smooth, proper orbifold $X$, and then to analyze the localization sequence arising from an affine surface contained in $X$. At the key step, we produce a Hodge class of rank $\per(\alpha)$ using de Jong's theorem \cite{dejong_period_index} that $\per(\alpha) = \ind(\alpha)$. We mention a corollary:

	\begin{corollary}	
	\label{cor:birational_invariant}
	    The truth of the integral Hodge conjecture for $\dperf(X)$ is a birational invariant for smooth, proper Deligne--Mumford threefolds.
	\end{corollary}

	\begin{proof}
	    Combine Theorem~\ref{thm:ihc_for_dm_surfaces} with weak factorization (Example~\ref{ex:weak_factorization}).
	\end{proof}

	We now begin the proof of Theorem~\ref{thm:ihc_for_dm_surfaces}.

	\step{1} By Example~\ref{ex:total_rig}, it suffices to prove the integral Hodge conjecture for $\dperf^1(\scx)$, where $\scx \to X$ is a $\bmu_n$-gerbe over a smooth, proper orbifold.

	\step{2} Since the integral Hodge conjecture holds for smooth Deligne--Mumford stacks of dimension at most $1$ by Lemma~\ref{lem:ihc_curve_surjective}, weak factorization (Example~\ref{ex:weak_factorization}) implies that the integral Hodge conjecture for $\dperf^1(\scx)$ is a birational invariant of $\scx$.

	\step{3} Let $U \subset X$ be an open substack equivalent to an affine surface. Possibly shrinking $U$, we may assume that $\NS(U) = 0$. We observe that since $U$ is an affine surface, the restriction $\scu \to U$ of $\scx$ over $U$ is essentially topologically trivial.

	From the previous step, we may replace $\scx$ by any other smooth compactification of $\scu$. In particular, following the proof of Lemma~\ref{lem:open_subvariety_gerbe}, we adopt the following setup: $U \to X_0$ is the inclusion of $U$ into a smooth, projective variety such that $D = X_0 - U$ has simple normal crossings; $X \to X_0$ is an iterated root stack over the components of $D$; $\scx \to X$ is a $\bmu_n$-gerbe extending $\scu$; and $\P \to X$ is a Severi--Brauer variety representing the Brauer class of $\scx$. 

	\step{4} There is a commutative diagram
	\[
		\begin{tikzcd}
			\Hdg(\dperf(\P)_1, \Z) \ar[d, "\sim"] \ar[r, "f"] & \Hdg(\wtop_0(\dperf(\P_U)_1) \ar[r] & \ktop_0(\dperf(\P_U)_1) \ar[d, "\sim"] \\
			\Hdg(\dperf^1(\scx), \Z) \ar[rr, "g"] & & \ktop_0(\dperf^1(\scu)),
		\end{tikzcd}
	\]
	where the vertical isomorphisms come from Lemma~\ref{lem:weighted_lb_lemma}, after choosing $\sco_{\P_{\scx}}(1) \in \Pic^{-1}(\P_{\scx})_{1}$. 

	The key claim is that any fiber of $g$ contains an algebraic class. To prove the claim, it suffices to prove the corresponding statement for $f$. Observe that by Lemma~\ref{lem:affine_surface} and the assumption that $\NS(U) = 0$, the rank map gives an isomorphism
	\[
		\Hdg(\wtop_0(\dperf(\P_U)_1) \simeq \per(\alpha) \cdot \Z,
	\]
	where $\alpha$ is the Brauer class of $\scx$. Therefore, it suffices to show that there is an algebraic class of rank $\per(\alpha)$ in $\Hdg(\dperf^1(\scx), \Z)$. This follows from de Jong's theorem \cite{dejong_period_index} that $\per(\alpha) = \ind(\alpha)$.

	\step{5} Let $E$ be the preimage of $D$ under the map $\scx \to X_0$. From the localization sequence (Lemma~\ref{lem:excision_sequence}), we may consider a diagram with exact rows
	\[
		\begin{tikzcd}
			K \ar[r] \ar[d] & \Hdg(\dperf^1(\scx), \Z) \ar[d] \ar[r] & \ktop_0(\dperf^1(\scu)) \ar[d, equals] \\
			\ktop_0(\dcoh(E)) \ar[r] & \ktop_0(\dperf^1(\scx)) \ar[r] & \ktop_0(\dperf^1(\scu)),
		\end{tikzcd}
	\]
	where the vertical maps are inclusions. Our goal is to show that any $v \in \Hdg(\dperf^1(\scx), \Z)$ is algebraic. From the result of the previous step, we may suppose that $v$ lies in the image of $K$. Then we apply Lemma~\ref{lem:dm_curve}, which asserts that the map $\K_0(\dcoh(E)) \to \ktop_0(\dcoh(E))$ is surjective. \qed

	\bibliography{mybib}

\providecommand{\bysame}{\leavevmode\hbox to3em{\hrulefill}\thinspace}
\providecommand{\MR}{\relax\ifhmode\unskip\space\fi MR }
\providecommand{\MRhref}[2]{%
  \href{http://www.ams.org/mathscinet-getitem?mr=#1}{#2}
}
\providecommand{\href}[2]{#2}
\begin{thebibliography}{BZFN12}

\bibitem[Ada61]{adams}
J.~F. Adams, \emph{On {C}hern characters and the structure of the unitary
  group}, Proc. Cambridge Philos. Soc. \textbf{57} (1961), 189--199.
  \MR{121795}

\bibitem[AG14]{Antieau-Gepner}
Benjamin Antieau and David Gepner, \emph{Brauer groups and \'{e}tale cohomology
  in derived algebraic geometry}, Geom. Topol. \textbf{18} (2014), no.~2,
  1149--1244. \MR{3190610}

\bibitem[AH62]{atiyah_hirz}
M.~F. Atiyah and F.~Hirzebruch, \emph{Analytic cycles on complex manifolds},
  Topology \textbf{1} (1962), 25--45. \MR{145560}

\bibitem[AH18]{antieau_heller}
Benjamin Antieau and Jeremiah Heller, \emph{Some remarks on topological
  {$K$}-theory of dg categories}, Proc. Amer. Math. Soc. \textbf{146} (2018),
  no.~10, 4211--4219. \MR{3834651}

\bibitem[Ant11a]{antieau_divisibility}
Benjamin Antieau, \emph{{\v{C}}ech approximation to the {Brown}-{Gersten}
  spectral sequence}, Homology Homotopy Appl. \textbf{13} (2011), no.~1,
  319--348 (English).

\bibitem[Ant11b]{antieau_et_ind}
\bysame, \emph{Cohomological obstruction theory for {Brauer} classes and the
  period-index problem}, J. \(K\)-Theory \textbf{8} (2011), no.~3, 419--435
  (English).

\bibitem[Ant17]{Antieau-et_tw}
\bysame, \emph{\'{E}tale twists in noncommutative algebraic geometry and the
  twisted {B}rauer space}, J. Noncommut. Geom. \textbf{11} (2017), no.~1,
  161--192. \MR{3626560}

\bibitem[AOV08]{tame_stacks_AOC}
Dan Abramovich, Martin Olsson, and Angelo Vistoli, \emph{Tame stacks in
  positive characteristic}, Ann. Inst. Fourier (Grenoble) \textbf{58} (2008),
  no.~4, 1057--1091. \MR{2427954}

\bibitem[AS06]{atiyah_segal}
Michael Atiyah and Graeme Segal, \emph{Twisted {$K$}-theory and cohomology},
  Inspired by {S}. {S}. {C}hern, Nankai Tracts Math., vol.~11, World Sci.
  Publ., Hackensack, NJ, 2006, pp.~5--43. \MR{2307274}

\bibitem[AW13]{antieau_williams_godeaux}
Benjamin Antieau and Ben Williams, \emph{Godeaux-{Serre} varieties and the
  {\'e}tale index}, J. \(K\)-Theory \textbf{11} (2013), no.~2, 283--295
  (English).

\bibitem[AW14]{antieau_williams_6_complex}
\bysame, \emph{The topological period-index problem over 6-complexes}, J.
  Topol. \textbf{7} (2014), no.~3, 617--640. \MR{3252958}

\bibitem[AW21]{top_per_ind}
\bysame, \emph{The topological period-index conjecture}, Math. Res. Lett.
  \textbf{28} (2021), no.~5, 1307--1317. \MR{4471709}

\bibitem[BCC92]{trento}
E.~Ballico, F.~Catanese, and C.~Ciliberto (eds.), \emph{Classification of
  irregular varieties}, Lecture Notes in Mathematics, vol. 1515,
  Springer-Verlag, Berlin, 1992, Minimal models and abelian varieties.
  \MR{1180332}

\bibitem[Ber09]{bernardara}
Marcello Bernardara, \emph{A semiorthogonal decomposition for {Brauer}-{Severi}
  schemes}, Math. Nachr. \textbf{282} (2009), no.~10, 1406--1413 (English).

\bibitem[Bla16]{blanc}
Anthony Blanc, \emph{Topological {K}-theory of complex noncommutative spaces},
  Compos. Math. \textbf{152} (2016), no.~3, 489--555. \MR{3477639}

\bibitem[BLS16]{geometricity_dm}
Daniel Bergh, Valery~A. Lunts, and Olaf~M. Schn\"{u}rer, \emph{Geometricity for
  derived categories of algebraic stacks}, Selecta Math. (N.S.) \textbf{22}
  (2016), no.~4, 2535--2568. \MR{3573964}

\bibitem[BO21]{two_coniveau}
Olivier Benoist and John~Christian Ottem, \emph{Two coniveau filtrations}, Duke
  Math. J. \textbf{170} (2021), no.~12, 2719--2753. \MR{4305380}

\bibitem[BR19]{bergh_weak_factor}
Daniel Bergh and David Rydh, \emph{Functorial destackification and weak
  factorization of orbifolds}, 2019, arXiv:1905.00872.

\bibitem[BS21]{bergh_gerbe}
Daniel Bergh and Olaf~M. Schn\"{u}rer, \emph{Decompositions of derived
  categories of gerbes and of families of {B}rauer-{S}everi varieties}, Doc.
  Math. \textbf{26} (2021), 1465--1500. \MR{4334847}

\bibitem[BZFN10]{BZFN}
David Ben-Zvi, John Francis, and David Nadler, \emph{Integral transforms and
  {D}rinfeld centers in derived algebraic geometry}, J. Amer. Math. Soc.
  \textbf{23} (2010), no.~4, 909--966. \MR{2669705}

\bibitem[BZFN12]{benzvi2012morita}
David Ben-Zvi, John Francis, and David Nadler, \emph{Morita equivalence for
  convolution categories: Appendix to arxiv:0805.0157}, 2012.

\bibitem[CG20]{spinc}
Diarmuid Crowley and Mark Grant, \emph{The topological period-index conjecture
  for {${\rm spin}^c$} 6-manifolds}, Ann. K-Theory \textbf{5} (2020), no.~3,
  605--620. \MR{4132748}

\bibitem[CT02]{CT_gabber}
Jean-Louis Colliot-Th\'{e}l\`ene, \emph{Exposant et indice d'alg\`ebres simples
  centrales non ramifi\'{e}es}, Enseign. Math. (2) \textbf{48} (2002), no.~1-2,
  127--146, With an appendix by Ofer Gabber. \MR{1923420}

\bibitem[CT06]{ct_bourbaki}
\bysame, \emph{Alg\`ebres simples centrales sur les corps de fonctions de deux
  variables (d'apr\`es {A}. {J}. de {J}ong)}, no. 307, 2006, S\'{e}minaire
  Bourbaki. Vol. 2004/2005, pp.~Exp. No. 949, ix, 379--413. \MR{2296425}

\bibitem[CTV12]{CT_Voisin}
Jean-Louis Colliot-Th{\'e}l{\`e}ne and Claire Voisin, \emph{Cohomologie non
  ramife{\'e} et conjecture du {Hodge} enti{\`e}re}, Duke Math. J. \textbf{161}
  (2012), no.~5, 735--801 (French).

\bibitem[Del71]{deligne_hodge_ii}
Pierre Deligne, \emph{Th{\'e}orie de {Hodge}. {II}}, Publ. Math., Inst. Hautes
  {\'E}tud. Sci. \textbf{40} (1971), 5--57 (French).

\bibitem[DG13]{drinfeld_gaitsgory}
Vladimir Drinfeld and Dennis Gaitsgory, \emph{On some finiteness questions for
  algebraic stacks}, Geom. Funct. Anal. \textbf{23} (2013), no.~1, 149--294.
  \MR{3037900}

\bibitem[dJ04]{dejong_period_index}
A.~J. de~Jong, \emph{The period-index problem for the {B}rauer group of an
  algebraic surface}, Duke Math. J. \textbf{123} (2004), no.~1, 71--94.
  \MR{2060023}

\bibitem[Gir71]{giraud}
Jean Giraud, \emph{Cohomologie non ab\'{e}lienne}, Die Grundlehren der
  mathematischen Wissenschaften, Band 179, Springer-Verlag, Berlin-New York,
  1971. \MR{0344253}

\bibitem[Gra04]{grabowski}
Craig Grabowski, \emph{On the integral {H}odge conjecture for 3-folds},
  ProQuest LLC, Ann Arbor, MI, 2004, Thesis (Ph.D.)--Duke University.
  \MR{2707508}

\bibitem[Har17]{harper}
Alicia Harper, \emph{Factorization for stacks and boundary complexes}, 2017,
  arXiv:1706.07999.

\bibitem[HLP20]{equiv_hodge}
Daniel Halpern-Leistner and Daniel Pomerleano, \emph{Equivariant {H}odge theory
  and noncommutative geometry}, Geom. Topol. \textbf{24} (2020), no.~5,
  2361--2433. \MR{4194295}

\bibitem[HR17]{rydh_hall}
Jack Hall and David Rydh, \emph{Perfect complexes on algebraic stacks}, Compos.
  Math. \textbf{153} (2017), no.~11, 2318--2367. \MR{3705292}

\bibitem[HS05]{huybrechts_stellari}
Daniel Huybrechts and Paolo Stellari, \emph{Equivalences of twisted {$K3$}
  surfaces}, Math. Ann. \textbf{332} (2005), no.~4, 901--936. \MR{2179782}

\bibitem[Kal08]{kaledin}
D.~Kaledin, \emph{Non-commutative {H}odge-to-de {R}ham degeneration via the
  method of {D}eligne-{I}llusie}, Pure Appl. Math. Q. \textbf{4} (2008), no.~3,
  Special Issue: In honor of Fedor Bogomolov. Part 2, 785--875. \MR{2435845}

\bibitem[Kre03]{kresch_quarternion}
Andrew Kresch, \emph{Hodge-theoretic obstruction to the existence of quaternion
  algebras}, Bull. London Math. Soc. \textbf{35} (2003), no.~1, 109--116.
  \MR{1934439}

\bibitem[KV04]{vistoli_kresch}
Andrew Kresch and Angelo Vistoli, \emph{On coverings of {D}eligne-{M}umford
  stacks and surjectivity of the {B}rauer map}, Bull. London Math. Soc.
  \textbf{36} (2004), no.~2, 188--192. \MR{2026412}

\bibitem[Lie07]{lieblich_moduli}
Max Lieblich, \emph{Moduli of twisted sheaves}, Duke Math. J. \textbf{138}
  (2007), no.~1, 23--118 (English).

\bibitem[Lie08]{lieblich_twisted_sheaves}
\bysame, \emph{Twisted sheaves and the period-index problem}, Compos. Math.
  \textbf{144} (2008), no.~1, 1--31. \MR{2388554}

\bibitem[Lie11]{lieblich_arithmetic_surface}
\bysame, \emph{Period and index in the {B}rauer group of an arithmetic
  surface}, J. Reine Angew. Math. \textbf{659} (2011), 1--41, With an appendix
  by Daniel Krashen. \MR{2837009}

\bibitem[Mat16]{matzri}
Eliyahu Matzri, \emph{Symbol length in the {B}rauer group of a field}, Trans.
  Amer. Math. Soc. \textbf{368} (2016), no.~1, 413--427. \MR{3413868}

\bibitem[Mou19]{moulinos}
Tasos Moulinos, \emph{Derived {A}zumaya algebras and twisted {$K$}-theory},
  Adv. Math. \textbf{351} (2019), 761--803. \MR{3955575}

\bibitem[Orl16]{orlov_gluing}
Dmitri Orlov, \emph{Smooth and proper noncommutative schemes and gluing of {DG}
  categories}, Adv. Math. \textbf{302} (2016), 59--105. \MR{3545926}

\bibitem[Per19]{Perry_NCHPD}
Alexander Perry, \emph{Noncommutative homological projective duality}, Adv.
  Math. \textbf{350} (2019), 877--972. \MR{3948688}

\bibitem[Per22]{Perry_CY2}
\bysame, \emph{The integral {H}odge conjecture for two-dimensional
  {C}alabi-{Y}au categories}, Compos. Math. \textbf{158} (2022), no.~2,
  287--333. \MR{4406785}

\bibitem[SdJ10]{dejong_starr}
Jason Starr and Johan de~Jong, \emph{Almost proper {GIT}-stacks and
  discriminant avoidance}, Doc. Math. \textbf{15} (2010), 957--972.
  \MR{2745688}

\bibitem[Shi21]{shin}
Minseon Shin, \emph{The cohomological {B}rauer group of a torsion
  {$\mathbf{G}_m$}-gerbe}, Int. Math. Res. Not. IMRN (2021), no.~19,
  14480--14507. \MR{4387781}

\bibitem[Tho85]{thomason}
R.~W. Thomason, \emph{Algebraic {$K$}-theory and \'{e}tale cohomology}, Ann.
  Sci. \'{E}cole Norm. Sup. (4) \textbf{18} (1985), no.~3, 437--552.
  \MR{826102}

\bibitem[To{\"{e}}12]{Toen_der_azu}
Bertrand To{\"{e}}n, \emph{Derived {A}zumaya algebras and generators for
  twisted derived categories}, Invent. Math. \textbf{189} (2012), no.~3,
  581--652. \MR{2957304}

\end{thebibliography}
	\bibliographystyle{amsalpha}

\end{document}